\documentclass[11pt]{amsart}

\usepackage{amssymb, amsmath, amsthm}

\usepackage{epsfig}
\usepackage{graphicx}
\usepackage{color}
\usepackage[alphabetic]{amsrefs}
\usepackage{mathptmx}

\usepackage{hyperref}
\usepackage{fullpage}
\usepackage{pinlabel}
\usepackage{wrapfig}

\theoremstyle{plain}
\newtheorem*{theorem*}{Theorem}
\newtheorem{theorem}{Theorem}[section]
\newtheorem{definition}[theorem]{Definition}
\newtheorem{lemma}[theorem]{Lemma}
\newtheorem{corollary}[theorem]{Corollary}
\newtheorem{example}[theorem]{Example}
\newtheorem{remark}[theorem]{Remark}
\newtheorem{question}[theorem]{Question}
\newtheorem{conjecture}[theorem]{Conjecture}

\newcommand{\R}{{\mathbb R}}

\newcommand{\N}{\mathbb N}

\newcommand{\nil}{\varnothing}

\newcommand{\defn}[1]{\textbf{#1}}
\newcommand{\boundary}{\partial}
\newcommand{\mc}[1]{\mathcal{#1}}

\newcommand{\tr}{\operatorname{tr}}

\newcommand{\weight}{\omega}
\newcommand{\width}{\Omega}

\newcommand{\co}{\mskip0.5mu\colon\thinspace}
\newcommand{\spacing}{ \parskip 6.6pt \parindent 0pt}

\spacing

\begin{document}

   \title{Combinatorial minimal surfaces in pseudomanifolds and other complexes}
   \author{Weiyan Huang, Daniel Medici, Nick Murphy, Haoyu Song, Scott A. Taylor, Muyuan Zhang}

 \begin{abstract}
We define combinatorial analogues of stable and unstable minimal surfaces in the setting of weighted pseudomanifolds and their generalizations. We prove that, under mild conditions, such combinatorial minimal surfaces always exist. We use a technique, adapted from work of Johnson and Thompson, called \emph{thin position}. Thin position is defined using orderings of the $n$-dimensional simplices of an $n$-dimensional pseudomanifold. In addition to defining and finding combinatorial minimal surfaces, from thin orderings, we derive invariants of even-dimensional closed simplicial pseudomanifolds called \emph{width} and \emph{trunk}. We study the additivity properties of these invariants under connected sum and prove theorems analogous to theorems in knot theory and 3-manifold theory.
\end{abstract}



   \maketitle

 \section{Introduction}
 
Given a Riemannian $n$-manifold $M$, a minimal surface is a codimension 1 submanifold $U \subset M$ that is a critical point of the area (i.e. $(n-1)$-dimensional volume) functional on all compactly supported variations of $U$ \cite[Definition 2.1.4]{MeeksPerez}. The quest of discrete differential geometry is to convert definitions and theorems from differential geometry into definitions and results concerning discrete objects, such as simplices and triangulations. Typically, these discrete objects inherit a geometry from an embedding in some (typically low-dimensional) Euclidean space. In this paper, we embark on the related (but different) quest to convert differential geometric ideas into combinatorial analogues. We adapt ideas of Thompson \cite{Thompson} and Johnson \cite{Johnson} to produce examples of \emph{combinatorial} minimal surfaces in weighted pseudomanifolds.

Deferring our definitions until Sections \ref{Combinatorial Minimal Surfaces} and  \ref{orderings}, we:
\begin{itemize}
\item Show how an ordering $\mc{O}$ of the top-dimensional cells of an $n$-dimensional pseudomanifold having non-negative weights on the $(n-1)$ faces can be put into a ``local thin position'' $\mc{O}'$. (Definition \ref{def: loc thin})
\item Prove that in a locally thin ordering, local minima (of a certain function) are stable combinatorial minimal surfaces and local maxima are either stable or unstable combinatorial minimal surfaces (Theorem \ref{Main Theorem}).
\item Use thin position to define invariants (called \emph{width} and \emph{trunk}) of simplicial pseudomanifolds. We study the additivity properties of these invariants under connected sum and prove theorems analogous to theorems in knot theory and 3-manifold theory.
\item Pose a list of conjectures and questions inspired by results in knot theory and 3-manifold theory.
\end{itemize}

We begin with background on our motivation and techniques. 

\subsection{Background}
Given a concept in differential geometry, there may be several ways of defining a combinatorial analogue. For instance, given a triangulated surface,  a ``curve'' might be considered to be either a cycle in the 1-skeleton of the triangulation (as in \cite{Thompson}) or as a normal curve transverse to the triangulation (as in \cite{Matveev}). Similarly, given a triangulation of a 3-manifold, a surface might be considered either as a subcomplex of the triangulation or as a certain kind of normal surface (as in \cite{JR}). When deciding what additional criteria should be imposed on a curve or surface in order to be considered a geodesic or minimal surface, certain aspects of the differential theory are selected to be of primary importance. For instance, least weight normal surfaces function much as least area surfaces in differential geometry (see \cite{JR}). Similarly, in \cites{Bobenko, BS}, the isothermic properties of minimal surfaces are chosen as the key feature to adapt to the discrete geometric setting. In this paper, we elect to adapt the notion of minimal surface to the combinatorial setting by adapting Pitts' construction \cite{Pitts} of differential geometric minimal surfaces using sweepouts. 

Consider orientable surfaces in a closed orientable 3-manifold endowed with a Riemannian metric. One view of minimal surfaces is that they are critical points of certain functionals on parameterized families of surfaces; the number of parameters used to define the family should correspond to the index of the minimal surface. In 3-manifold topology, this view of minimal surfaces has been reinterpreted in terms of compressing discs for surfaces (see \cite{Pitts} and \cite{Bachman} to begin). A surface without compressing discs (and which is not a 2-sphere bounding a 3-ball) is the topological analogue of a stable minimal surface. A surface which has compressing discs on both sides, but which does not have a pair of disjoint compressing discs on opposite sides, (a \defn{weakly incompressible} surface) is the analogue of an index 1 minimal surface. When a weakly incompressible surface is in an irreducible 3-manifold and can be compressed down to trivial 2-spheres in both directions, it is called a \defn{strongly irreducible Heegaard surface}. Recently Ketover and Liokumovich \cite{KetoverLiokumovich} confirmed a conjecture of Pitts and Rubinstein \cite{Rub} by showing that these ``topological minimal surfaces'' are equivalent (e.g. isotopic) to ``geometric minimal surfaces'' of index 0 or 1. Inspired by Rubinstein's view of topological notions of minimal surfaces, Bachman \cite{Bachman} defined a topological index for surfaces in 3-manifolds in terms of the compressing discs for those 3-manifolds. 

In a different direction, Thompson considered cycles in the 1-skeleton of a triangulated 2--sphere. She identified certain ``shortening moves'' as the combinatorial analogue of both the geometric notion of creating a parameterized family of curves and the topological notion of a compressing disc for a surface in the 3-manifold. More precisely, in Thompson's work, if $\mc{T}$ is a (simplicial) triangulation of a surface and if $\gamma$ is a path in the 1-skeleton of $\mc{T}$, then a triangle $T \in \mc{T}$ is a \defn{shortening move} for $\gamma$ if two adjacent edges of $\gamma$ are two sides $a,b$ of $T$ and $\gamma$ does not contain the third side $c$ of $T$. See Figure \ref{fig:pathtriangulateddisc} for an example. Replacing $a,b$ in $\gamma$ with $c$ produces a new path $\gamma'$ containing one less edge than $\gamma$. For Thompson, a \defn{stable closed geodesic} is a cycle in the 1-skeleton of $\mc{T}$ which does not have a shortening move and which is not the boundary of a single triangle. A \defn{closed unstable geodesic}, on the other hand, is a cycle $\gamma$ in the 1-skeleton which has shortening moves, at least one on each side of $\gamma$, but for which any two shortening moves on opposite sides of $\gamma$ share an edge. The path in Figure \ref{fig:pathtriangulateddisc} is an unstable geodesic in the triangulated sphere obtained by doubling the disc along its boundary. Thompson uses a technique called ``thin position'' to produce stable and unstable closed geodesics in $\mc{T}$.

\begin{figure}
\centering
\includegraphics[scale=0.5]{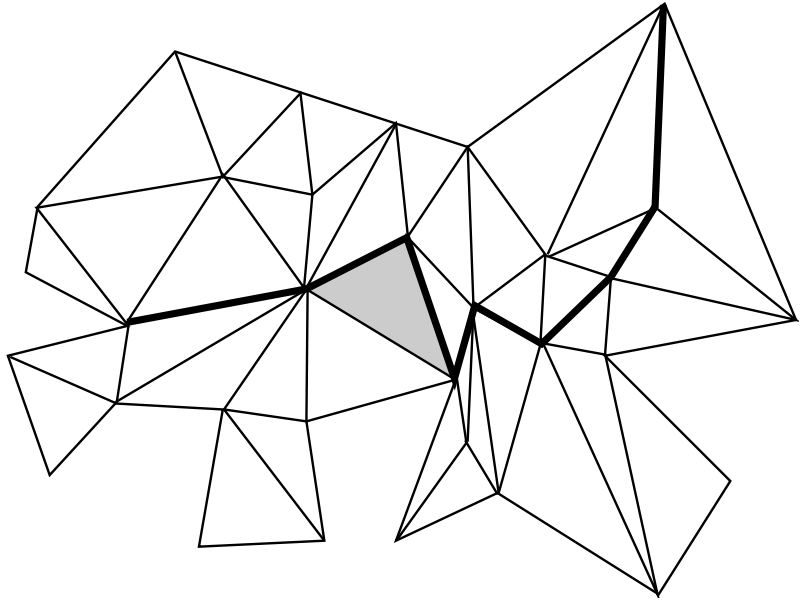}
\caption{A path in a triangulated disc is indicated with bold lines. The shaded triangle is a shortening move for the path, as replacing the two edges of the path in the triangle with the edge of the path not in the triangle creates a new path containing fewer edges.}
\label{fig:pathtriangulateddisc}
\end{figure}

The main point of this paper is to generalize Thompson's ideas to weighted pseudomanifolds $(M, \weight)$ of all dimensions. An $n$-dimensional simplicial manifold where each $(n-1)$-face is given a weight $\weight = 1$, is an example of a weighted pseudomanifold. If $M$ is a triangulated Riemannian manifold, it is also natural to consider the weight function to be the $(n-1)$-volume of each $(n-1)$-dimensional face. In our context, however, the weight function need not have any actual geometric meaning attached to it. Additionally, all that matters for our construction is that the space $M$ is made up of $n$-dimensional spaces glued together along $(n-1)$-dimensional subspaces in their boundaries. 

Our work can also be seen as a version of Johnson's work \cite{Johnson} using thin position techniques to create a clustering algorithm.  Johnson considers graphs with weighted edges and uses thin position to produce subsets of the vertices (called ``pinch clusters'') that are more connected to each other than to vertices not in the subset. (See \cite{BHJ} for an improved version.) We can convert our setting to that of Johnson's by considering the dual graph to the $n$-dimensional simplices in $M$. That is, consider the graph where each vertex corresponds to an $n$-dimensional simplex and two vertices are joined by an edge with weight $w$ if they share an $(n-1)$-dimensional face of weight $w$. Johnson considers only the significance of the vertex subsets corresponding to our stable minimal surfaces. Since we also are interested in unstable minimal surfaces, we prefer to develop the theory from first principles in a way analogous to \cite{Thompson}. 

Our main technique, derived from \cite{Johnson} and \cite{Thompson}, is called ``thin position.'' Thin position was first defined in the setting of knots in the 3-sphere by Gabai \cite{Gabai}. Gabai used thin position to show that the local maxima of a certain function (called \emph{width}) had certain desirable properties, analogous to those of an unstable minimal surface. Thompson \cite{Thompson97} later showed that the global minimum of width also has certain desirable properties, analogous to those of a stable minimal surface. Thin position has since been adapted to numerous other settings by other authors. The most significant adaptation is likely that of Scharlemann and Thompson \cite{SchTh} who defined thin position for a 3-manifold. They show that the local minima and local maxima of width have analogous nice properties. 

Scharlemann and Thompson's idea is the main antecedent for the ideas in this paper. They consider handle structures on 3-manifolds and associate an ordering (of sorts) to each handle structure. An ordering is \emph{thinned} by swapping the order of a 3-dimensional 1-handle and 2-handle, so that the 2-handle is attached before the 1-handle. In our setting, the top-dimensional simplices of the pseudomanifold play the role of the handles in a handle decomposition. As the simplices are attached, the boundary of all the cells attached up to a certain point forms a level surface. The attachment of some simplices cause the weight of the level surface to increase and others to decrease. In contrast to the analogous situation with knots and 3-manifolds, some of the attachments may also cause the weight of the level surface to remain the same. Our basic aim (again following \cite{Johnson} and \cite{Thompson}) will be to swap adjacent cells in the ordering so that, if possible, the weight of the sublevel surface decreases before it increases.

 \section{Combinatorial minimal surfaces}\label{Combinatorial Minimal Surfaces}
 
 \begin{definition}\label{def: pseudomanifold}
An $n$-dimensional \defn{pseudomanifold} is a finite $n$-dimensional simplicial complex such that:
\begin{enumerate}
\item  it is the union of its $n$-simplices
\item (the \defn{pure} condition) each $(n-1)$-dimensional face is incident to exactly two $n$-dimensional faces
\item (the \defn{strongly connected} condition) Let $\Gamma$ be the graph having  a vertex for each $n$-dimensional simplex and with an edge joining two vertices if and only if the corresponding simplices share an $(n-1)$-face. Then $\Gamma$ is connected.
\end{enumerate}
\end{definition}

For most of this paper, only the pure condition plays a significant role; consequently, we can work in greater generality. We will consider $n$-dimensional spaces built out of $n$-dimensional manifolds that we call ``$n$-bricks.'' The boundary of these bricks are subdivided into ``facets'' and bricks are glued along facets to make ``brick complexes.'' In what follows, we will often use $T$ to refer to an $n$-brick, inspired by the solid (T)riangles of a triangulated surface. The edges of the triangle are its facets.

 More formally, a single point (i.e. a 0-dimensional ball) is a \defn{$0$-brick}. For $n \geq 1$, an \defn{$n$-brick} is a compact $n$-dimensional manifold $T$ with $\boundary T$ tiled by finitely many $(n-1)$-dimensional bricks, called \defn{facets}. The recursive nature of this definition allows us to refer to the \defn{faces} of an $n$-brick. For instance, an $(n-2)$-dimensional face is a facet of one of the facets in the tiling of the boundary of the $n$-brick. 

It's most natural to consider the case when an $n$-brick is a compact $n$-ball, but other situations may also be useful. A \defn{gluing map} from an $n$-brick $T_1$ to a \emph{distinct} $n$-brick $T_2$ is a homeomorphism from the union $X$ of some facets of $T_1$ to the union $Y$ of some facets of $T_2$ such that the restriction of the map to each face in $X$ takes that face to a face in $Y$. A \defn{brick complex } $(M,\mc{T})$ consists of a topological space $M$, a finite collection of $n$-bricks $\mc{T}$ together with gluing maps such that:
 \begin{enumerate}
 \item For each $n$-brick $T \in \mc{T}$ and each facet $F$ of $T$, there is at most one gluing map $\phi$ such that $T$ is either in the domain or range of $\phi$.
 \item  $M$ is homeomorphic to the quotient space obtained by gluing together the $n$-brick in $\mc{T}$ using the gluing maps.
 \end{enumerate}

\begin{remark}
We will often refer to a brick complex  $(M,\mc{T})$ using only $M$. This obscures not only the gluing maps, but also bricks. However, in this paper, no confusion should result. We also refer to the images in $M$ of the $n$-bricks as \defn{bricks} and the images in $M$ of the facets as \defn{facets}.
\end{remark} 

Natural examples of brick complexes include pseudomanifolds, and their generalizations pure simplicial complexes and pure polytopal complexes \cite{Ziegler} that are the union of their $n$-cells. Unlike in some definitions of \emph{pseudomanifold} we do not allow a cell to be incident to itself along a facet; allowing that in our setting would not change much, but adds some additional complication. Most of the important features of our work are easily visualized in the case when $n \in \{2,3\}$ and $(M,\mc{T})$ is a triangulated surface or 3-manifold.

\begin{example}
A compact 2-dimensional manifold with a triangulation is a brick complex. Each solid triangle is a brick and the edges of the triangulation are facets.
\end{example}

\begin{example}
A compact 3-manifold with a triangulation is a brick complex. Each solid tetrahedron is a brick and the faces of the triangulation are the facets.
\end{example}

\begin{example}
A non-compact finite-volume hyperbolic 3-manifold having an ideal triangulation such that no ideal tetrahedron has two of its faces identified produces a brick complex when each cusp is collapsed to a point.
\end{example}

\begin{example}
A compact 2-dimensional manifold with a pants decomposition is a brick complex, where each pair of pants is a brick and the cuffs of a pair of pants are the facets.
\end{example}

Let $\mc{T}$ denote the set of bricks in $M$ and $\mc{F}$ the set of facets in $M$. We denote by $\boundary M$ the set of facets adjacent to a single brick. If $A$ is a brick, we let $\boundary_I A$ denote the facets (called the \defn{interior facets}) of $A$ which do not lie in $\boundary M$. 

A \defn{surface} $S \subset M$ is the union of facets. We refer to the facets of $M$ belonging to $S$ as facets \defn{in} $S$.  The $(n-2)$-dimensional faces of $M$ belonging to $S$ are facets \defn{of} $S$. Observe that since facets are $(n-1)$-bricks, it would be natural to consider $S$ as a brick complex, except that more than two $(n-1)$-bricks may meet along a facet of $S$. A surface is \defn{proper} if no facet in $S$ is also a facet in $\boundary M$ and if for every facet $\alpha$ \emph{of} $S$ that does not lie in $\boundary M$, there are an even number of facets \emph{in} $S$ containing $\alpha$ as a facet. Informally, we can imagine $S$ as an immersed surface (in the topological sense) passing transversely through itself at $\alpha$.

If $S_1$ and $S_2$ are surfaces, we let $S_1 \sqcap S_2$ denote the facets in $S_1$ and $S_2$ common to both. We let $S_1 \setminus S_2$ be the facets in $S_1$ that are not facets of $S_2$. For a brick $A$ and a proper surface $S$, we let $\boundary_S A$ be the union of the facets of $\boundary_I A$ in $S$ and $\boundary'_S A$ the union of the facets of $\boundary_I A$ not in $S$.

As a combinatorial proxy for area, we will consider weight systems on brick complexes. A \defn{weight system} for a brick complex $M$ is a function $\weight \co \mc{F} \to [0,\infty) \subset \R$. If $S$ is the union of facets in $M$, we define the \defn{weight} $\weight(S)$ of $S$ to be the sum of the weights of the facets in $S$. 

\begin{example}
When $M$ is a compact 2-dimensional Riemannian manifold with a triangulation, it is natural to take the weight function to be the length of the facets (i.e. edges) in the boundary of each cell. Similarly, if $M$ is a simplicial or polytopal complex, it is natural to take the weight function to be the $(n-1)$-dimensional volume of each facet. Alternatively, we could take the weight of each facet simply to be equal to 1.
\end{example}

\begin{example}
When $M$ is a 2-dimensional Riemannian manifold with a pants decomposition, it is natural to take the weight function to be the length of each cuff of each pair of pants.
\end{example}

\begin{example}
Suppose that $\tau$ is a regular train track (see \cite{PH}) on a compact 2-dimensional manifold with a transverse measure. Let $N(\tau)$ be the fibered neighborhood of $\tau$. Decompose $N(\tau)$ into bricks by using fibers based at the midpoints of each edge. Each brick is a 2-dimensional disc whose boundary has been partitioned into arcs, which are the facets. Three of the arcs correspond to the fibers used in the decomposition; give those facets a weight corresponding to the measure of that branch of the train track. The other edges (i.e. facets) belong to the horizontal and vertical boundary of $N(\tau)$. 
\end{example}

\begin{example}
Figure \ref{Fig:samplebrickcplx} depicts a cell decomposition of a disc. Giving each edge a weight $\weight = 1$, we have a brick complex.
\end{example}

\begin{figure}[ht]
\labellist
\small\hair 2pt
\pinlabel {$1$} at 127 322
\pinlabel $2$ at  258 318
\pinlabel $3$ at 136 276
\pinlabel $4$ at 35 257
\pinlabel $5$ at 90 188
\pinlabel $6$ at 93 107
\pinlabel $8$ at 43 37
\pinlabel $9$ at 137 17
\pinlabel $10$ at 329 310
\pinlabel $11$ at 266 185
\pinlabel $12$ at 360 166
\pinlabel $13$ at 226 80
\pinlabel $14$ at 312 59
\pinlabel $15$ at 200 45
\pinlabel $16$ at 266 29
\pinlabel $7$ at 124 55
\pinlabel $17$ at 168 35
\endlabellist
\centering
\includegraphics[scale=0.5]{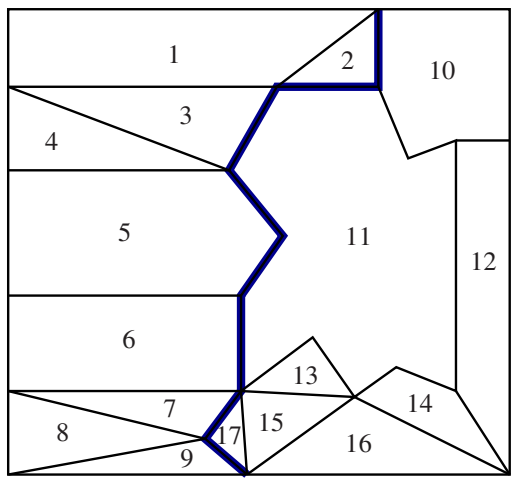}
\caption{An example of a cell complex homeomorphic to a disc. Giving each edge length 1 we have a brick complex. We have numbered each brick for convenience. The thick, blue path is an example of a proper surface $S$ of weight $\weight(S) = 8$.}
\label{Fig:samplebrickcplx}
\end{figure}

In what follows, we only require that the weights be non-negative; they need not arise from any geometric considerations. We call $(M, \omega)$ a \defn{weighted brick complex}. We will not be varying $\weight$; henceforth we denote a weighted brick complex simply by $M$.

\begin{definition}
The \defn{strength} of a brick $A$ relative to a surface $S$ is
\[
\sigma(A;S) = \weight(\boundary'_S A) - \weight(\boundary_S A).
\]
We can form a new surface $S_A$ by replacing $\boundary_S A \subset S$ with $\boundary'_S A$. We say that $A$ is a \defn{variation} for $S$ and that $S_A$ is obtained by \defn{varying} $A$. If $A \sqcap S \neq \nil$ and if $\sigma(A; S) \leq 0$, then $A$ is a \defn{shortening move} for $S$. If the inequality is strict, $A$ is a \defn{strict} shortening move for $S$.
\end{definition}

\begin{example}
Consider the brick complex $M$ in Figure \ref{Fig:samplebrickcplx}. The strength of brick 2 with respect to the highlighted surface (i.e. path) $S$ is $\sigma(2;S) = -1$ since it has two interior facets (i.e. edges) in $S$ and one interior facet that is not in $S$. The strength of brick 11 with respect to $S$ is $\sigma(11;S) =  7 - 5 = 2$. Varying $S$ across brick number 2 creates a surface of weight 7. Varying $S$ across brick number 11 creates a surface of weight 10. Varying $S$ across brick number 14 creates a disconnected surface of weight 12.
\end{example}

Observe that if $A$ is a variation of $S$, then $\weight(S_A) = \weight(S) + \sigma(S;A)$.  Thus, $A$ is a strict shortening move if and only if $\weight(S_A) < \weight(S)$. Notice also that if $S = \boundary \Delta$ for some brick $\Delta \subset M$, then $\Delta$ is a shortening move on $S$ and the empty surface results from shortening $S$ using $\Delta$. 

\begin{remark}
Johnson calls his version of $\sigma(A;S)$ (for graphs with weighted edges), the \emph{slope} of $A$ \cite{Johnson}.
\end{remark}

\begin{lemma}\label{still proper}
If $A$ is a variation of a proper surface $S$, then $S_A$ is a proper surface.
\end{lemma}

\begin{proof}
Assume that $S$ is a proper surface and that $A \subset M$ is a brick. By definition, no facet in $S$ or $\boundary_I A$ lies in $\boundary M$. Thus, no facet in $S_A$ lies in $\boundary M$.

Consider a facet $\alpha$ of $S_A$ (therefore, an $(n-2)$-brick) such that $\alpha$ does not lie in $\boundary M$. We will show that the number of facets in $S_A$, $\deg(\alpha; S_A)$,  containing $\alpha$ is even. Since $S$ is proper, $\deg(\alpha;S)$ is even (possibly 0). If $\alpha$ does not belong to $\boundary_I A$, then we are done, so suppose that $\alpha$ lies in $\boundary_I A$. Observe that $\boundary_I A$ is a proper surface tiled by facets. Thus, there exist precisely two facets $F$ and $F'$ in $\boundary_I A$ containing $\alpha$. If both $F$ and $F'$ lie in $S$, then $\deg(\alpha;S_A) = \deg(\alpha;S) - 2$ is still even. If precisely one of $F$ and $F'$ lies in $S$, then that facet does not lie in $S_A$, but the other one does. Hence, $\deg(\alpha;S_A) = \deg(\alpha;S)$ is still even. If neither $F$ nor $F'$ lie in $S$, then $\deg(\alpha;S_A) = \deg(\alpha;S) + 2$, and so $\deg(\alpha;S_A)$ is still even. Thus, $S_A$ is a proper surface.
\end{proof}

We now come to the central definition of this paper.

\begin{definition}\label{def:min surf}
A proper surface $S \subset M$ is:
\begin{itemize}
\item a \defn{stable minimal surface} if $S$ has no strict shortening moves,
\item an \defn{unstable minimal surface} if there is a partition of the bricks of $M$ into two non-empty sets $\mc{A}$ and $\mc{B}$ such that:
\begin{enumerate}
\item If $A \in \mc{A}$ and $B \in \mc{B}$, then $\boundary A \sqcap \boundary B \subset S$
\item Each of $\mc{A}$ and $\mc{B}$ contain a shortening move for $S$ and one of them contains a strict shortening move;
\item For each strict shortening move $A \in \mc{A}$ (resp. $B \in \mc{B}$), there exists a shortening move $B \in \mc{B}$ (resp. $A \in \mc{A}$) such that 
\[
2\weight(\boundary A \sqcap \boundary B) \geq |\sigma(A;S)| + |\sigma(B;S)|.
\]
\item For all strict shortening moves $A \in \mc{A}$ and strict shortening moves $B \in \mc{B}$,
\[
2\weight(\boundary A \sqcap \boundary B) \geq |\sigma(A;S)| + |\sigma(B;S)|.
\]
\end{enumerate}
\end{itemize}
\end{definition} 

In the context where $S$ is a smooth separating surface in a smooth manifold, the sets $\mc{A}$ and $\mc{B}$ correspond to the cells on either side of $S$. In Heegaard splitting theory, the sets $\mc{A}$ and $\mc{B}$ correspond to the sets of compressing discs on either side of a bicompressible (e.g. Heegaard) surface.

\begin{example}
In Figure \ref{Fig:samplebrickcplx}, the highlighted path $S$ is not a stable minimal surface since bricks 2 and 17 are strict shortening moves for $S$. Note also that brick 5 is a shortening move, but not a strict shortening move. These are the only shortening moves for $S$. Suppose that we have a partition of the bricks of $M$ into two nonempty sets $\mc{A}$ and $\mc{B}$ and that, without loss of generality, brick 2 is an element of $\mc{A}$. Since brick 2 shares no facets with either brick 5 or brick 17, $S$ is not an unstable minimal surface. After varying $S$ across both brick 2 and brick 17, we obtain a stable minimal surface $S'$. This is easily verified as all bricks other than 5 and 11 share at most one facet with $S'$. Brick 5 is a non-strict shortening move for $S'$ and Brick 11 is not a shortening move for $S'$.
\end{example}

\begin{figure}[ht!]
\centering
\includegraphics[scale=0.45]{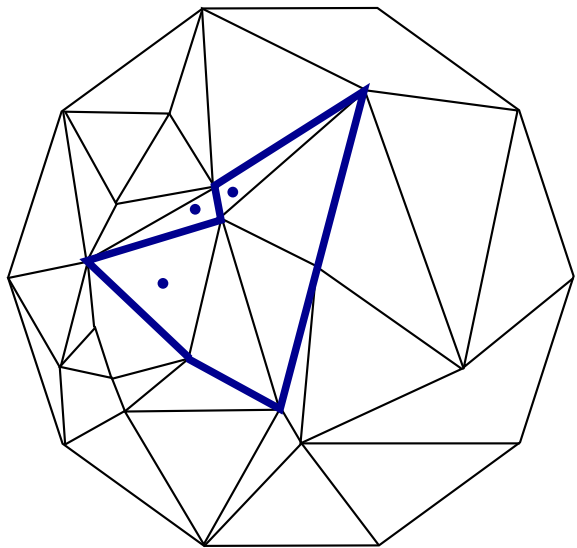}
\caption{A triangulated disc. If each edge is given weight 1, the blue curve is an example of an unstable minimal surface. There are three shortening moves, marked with dots.}
\label{fig:disctrian}
\end{figure}

\begin{example}
Figure \ref{fig:disctrian} shows a triangulated disc with a blue curve $S$. Take the sets $\mc{A}$ and $\mc{B}$ to consist of the triangles on the two sides of the curve and let $A$ and $B$ be the two of the marked shortening moves, chosen to be on opposite sides of the curve. Then 
\[
2\weight(\boundary A \sqcap \boundary B) = 2 = |\sigma(A;S)| + |\sigma(B;S)|.
\]
so Definition \ref{def:min surf} implies $S$ is an unstable minimal surface.
\end{example}

To better motivate this definition, consider a proper surface $S$ with two variations $A$ and $B$ such that $F = \boundary A \sqcap \boundary B \subset S$. Let $S_1 = S_A$ and $S_2 = S_B$. Then we say that $S_1, S, S_2$ is a \defn{variation sequence} for $S$. If $S$ is a stable minimal surface, then $S$ is a local minimum of $\weight$ for every variational sequence of $S$. 

Similarly, consider the surfaces
\[
\begin{array}{rcl}
S_{AB} &=& (S_A)_B\\
S_{BA} &=& (S_B)_A
\end{array}
\]
be the surfaces obtained by varying $S$ using $A$ and then $B$ and by varying $S$ using $B$ and then $A$. See Figure \ref{Fig:doublevariation} for an example where the surface is a curve in a triangulated disc.  Notice that $S_{AB} = S_{BA}$. We have
\[
\weight(S_{AB}) = \weight(S) + \sigma(A;S) + \sigma(B;S) + 2\weight(F).
\]
Suppose that $A$ and $B$ are both shortening moves for $S$. Then both $\sigma(A;S)$ and $\sigma(B;S)$ are non-positive. In which case $\weight(S_{AB}) < \weight(S)$ if and only if 
\[
|\sigma(A;S)| + |\sigma(B;S)| > 2\weight(F).
\]
Hence, if $S$ is an unstable minimal surface then $\weight(S_{AB}) \geq \weight(S)$. That is, $S$ is a local minimum of $\weight$ when varying using two variations.

\begin{figure}[ht]
\labellist
\small\hair 2pt
\pinlabel {$A$} at 127 588
\pinlabel {$B$} at 154 618
\pinlabel{Vary using $A$} [b] at 500 558
\pinlabel{Vary using $A$} [b] at 500 176
\pinlabel{Vary using $B$} [l] at 206 374
\pinlabel{Vary using $B$} [r] at 791 374
\pinlabel{$F$} [b] at 47 692
\endlabellist
\centering
\includegraphics[scale=0.45]{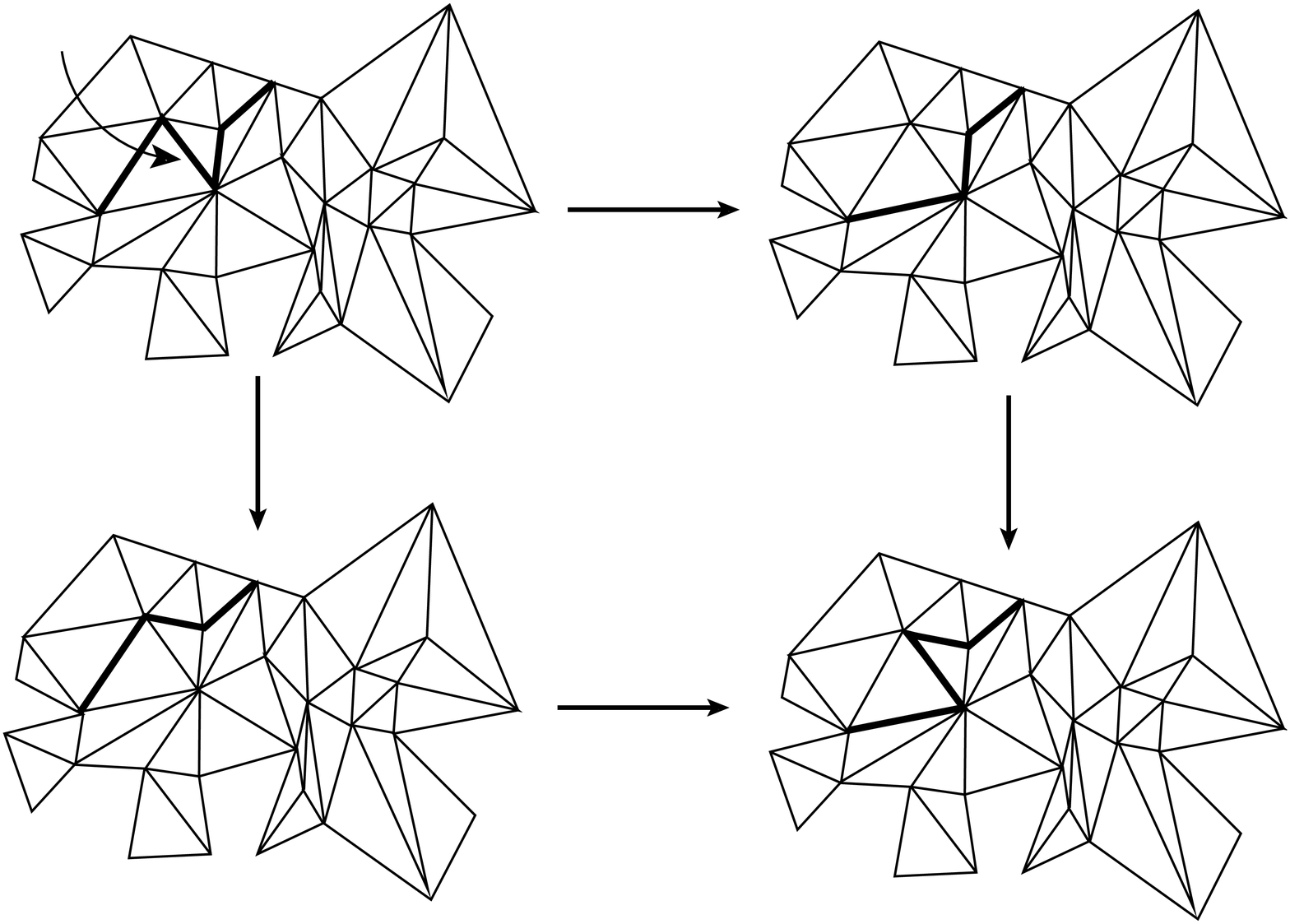}
\caption{Varying a surface (that is, a curve) in 2-dimensions using two variations. We can use arbitrary weights (not shown), but if we assign each edge a weight of 1, we have the following computations. The original curve has weight 4. The intersection $F = A \sqcap B$ is the edge shared by $A$ and $B$. Varying across either $A$ or $B$ produces a curve of weight 3, varying across both results in a curve of length 4. With respect to the original curve, both $A$ and $B$ have strength $-1$. After varying across $A$ or $B$, the strength of each of $A$ and $B$ is $+1$ with respect to the new curve. }
\label{Fig:doublevariation}
\end{figure}

\section{Orderings and Local Thin Position}\label{orderings}
 
 For $N \in \N\cup \{0\}$, we let $[N] = \{1, \hdots, N\}$ and $[N]^* = [N] \cup \{0\}$. Suppose that $M$ is a weighted brick complex with $N = |\mc{T}|$. An \defn{ordering} of $M$ is a function $\mc{O} \co \mc{T} \to [N]$. If the ordering $\mc{O}$ is clear from context, we let $T_i = \mc{O}^{-1}(i)$. Adapting terminology from Morse theory, for $j \in [N]^*$ we let 
 \[
 M_j = \bigcup\limits_{i \leq j} T_i
 \] 
be the \defn{sublevel set} of $\mc{O}$ at height $j$. Notice, $M_0 = \nil$. We let $M_j^C$ be the union of all the bricks $T_i$ with $i > j$. The \defn{level set} $S_j$ at height $j$ consists of all facets $F$ of $M_j \cap M^C_j$.   Let $\Lambda(j) = \weight(S_j)$. In the example from Figure \ref{Fig:samplebrickcplx}, the union of the bricks to the left of $S$ is the sublevel set at height 9, with $S = S_{9}$ being the level set at height 9.

The next lemma is where we use the fact that $M$ is a brick complex and not a general cell complex.
\begin{lemma}
For all $j \in [N]$, $S_j$ is obtained by varying $S_{j-1}$ across $T_j$. In particular, we have
\[
\Lambda(j) - \Lambda(j-1) = \sigma(T_j ;S_{i-1}).
\]
Also, each surface $S_j$ is a proper surface.
\end{lemma}
\begin{proof}
Notice that $S_0 = \nil$ is a proper surface with $\weight(S_0) = 0$. For $j \in [N]$, $M_j$ is obtained by including $T_j$ into $M_{j-1}$. Each facet of $\boundary_I T_j$ which is not also a facet of $M_{j-1}$ thus lies in $S_j$. Likewise, any facet of $\boundary_I T_j$ which does lie in $M_{j-1}$ lies in $S_{j-1}$ but not in $S_j$. Hence, $S_j$ is obtained by varying $S_{j-1}$ across $T_j$. The result follows from Lemma \ref{still proper}.
\end{proof}

A number $t \in [N-1]$ is a \defn{(local) maximum} for $\mc{O}$ if there exist $t_-, t_+ \in [N]^*$ such that $t_- < t < t_+$ and 
\begin{enumerate}
\item $\Lambda(t_-) < \Lambda(t_-+1)$, 
\item $\Lambda(t_+) < \Lambda(t_+-1)$, and
\item $\Lambda(t) = \Lambda(k)$ for all $k \in \{t_- +1, \hdots, t_+ -1\}$.
\end{enumerate}
We define \defn{(local) minima} similarly, reversing the inequalities in (1) and (2). The maximum or minimum is \defn{extremal} if $t = t_- +1$ or $t = t_+ -1$. 

\begin{example}
In the Example from Figure \ref{Fig:samplebrickcplx}, the level sets at heights 9 and 11 both have weight 8, while the level set at height 10 has weight 10, so 10 is a local maximum for the ordering.
\end{example}

Observe that $\Lambda(\nil) = \Lambda(N) = 0$ since both $S_0$ and $S_N$ are empty. Consequently, we refer to ``maxima'' and ``minima'' rather than ``local maxima'' and ``local minima.'' Also observe that if at least one weight on an interior facet is positive, then every ordering will have a maximum. 

\begin{definition}[Thompson]
The \defn{width} $\width(\mc{O})$ of an ordering $\mc{O}$ is the sequence whose terms are the values of $\Lambda$ at the maxima of $\mc{O}$, arranged in non-increasing order (with repetitions). Widths of orderings are compared lexicographically.
\end{definition}

\begin{remark}
Our definition of width is due to Thompson \cite{Thompson}. In \cite{Johnson}, Johnson defines \emph{width} using all the values of $\Lambda$, not only the maxima.  In that paper, he is primarily concerned with the properties of local minima of width. Using Thompson's width allows us, like Thompson, to consider both local maxima and local minima.
\end{remark}

In the remainder of the section, we consider the effect of interchanging two bricks in an ordering. We establish some notation. In the group of permutations of $[N]$, let $\tau_i$ be the transposition interchanging $i$ and $i+1$, for $i < N$. Suppose $\mc{O}$ is an ordering and that $i \in [N-1]$. Let $\mc{O}' = \tau_i \circ \mc{O}$. Set $A = \mc{O}^{-1}(i)$ and $B = \mc{O}^{-1}(i+1)$. Let $F = \boundary A \sqcap \boundary B$. For any $j$, let $S_j$ and $S'_j$ be the level sets at height $j$ for $\mc{O}$ and $\mc{O}'$ respectively. Observe that $S'_j = S_j$ whenever $j \leq i-1$ or $j\geq i+1$. The surface $S'_i$ is the variation of $S_{i-1}$ along $B$ and $S_{i+1}$ is the variation of $S'_i$ along $A$ and of $S_i$ along $B$. Let $\Lambda'$ be the function used to define width for $\mc{O}'$. 

Observe that:
\[
\Lambda'(i) = \Lambda(i) - \sigma(A;S_{i-1}) + \sigma(B;S_{i-1}).
\]

Since $\boundary_{S_{i+1}} B = \boundary_{S_{i-1}} B \cup F$, and since $F \sqcap S_{i-1} = \nil$, we have
\[
\sigma(B;S_i) = \sigma(B;S_{i-1}) - 2\weight(F).
\]
Observe also that
\[
\sigma(A;S_i) = - \sigma(A;S_{i-1})
\]
since each facet of $\boundary_I A$ belongs either to $S_{i-1}$ or to $S_i$. Thus,
\[
\Lambda'(i) = \Lambda(i) + \sigma(A;S_i) + \sigma(B;S_{i}) + 2\weight(F).
\]

As in Figure \ref{Fig: swaplemma}, there are times when swapping the order of two adjacent bricks can decrease width.

\begin{lemma}[The swap lemma]\label{swapping lem}
If 
\[
\sigma(A;S_i) + \sigma(B;S_{i}) + 2\weight(F) \leq  0 \hspace{.3in} (*)
\]
then $\width(\mc{O}') \leq \width(\mc{O})$. Furthermore, if $i$ is a maximum for $\mc{O}$ and if $(*)$ is strict, then $\width(\mc{O}') < \width(\mc{O})$.
\end{lemma}

\begin{figure}[ht!]
\labellist
\small\hair 2pt
\pinlabel {$i$} at 128 391
\pinlabel {$i+1$} at 210 391

\pinlabel {$i$} at 800 391
\pinlabel {$i+1$} at 720 391

\pinlabel {$i$} at 332 125
\pinlabel {$i+1$} at 137 90

\pinlabel {$i$} at  727 82
\pinlabel {$i+1$} at 925 117
\endlabellist
\includegraphics[scale=0.43]{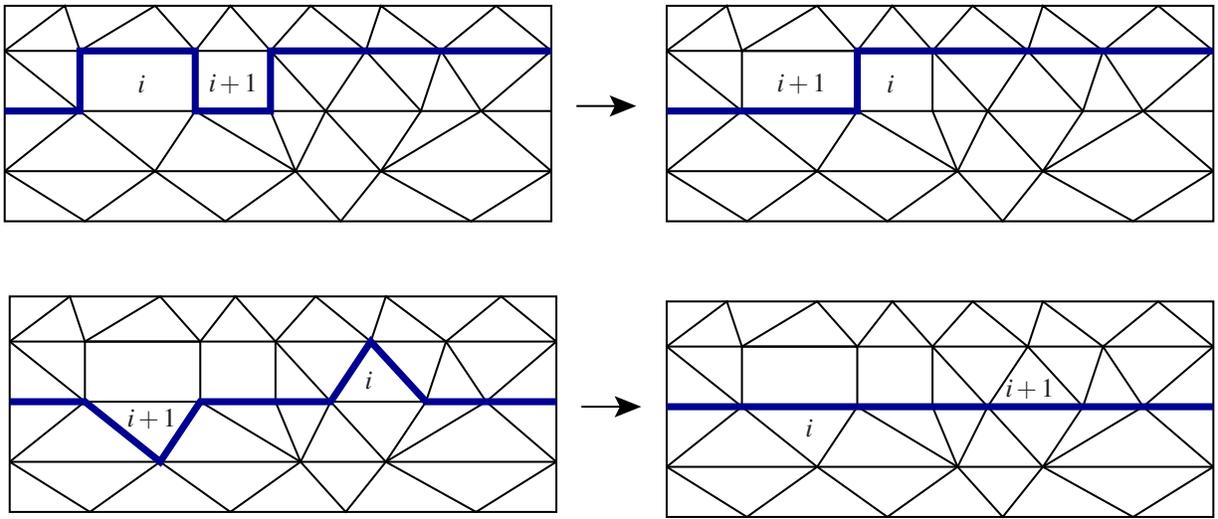}
\caption{Consider the pictured cell decomposition of a disc. Give each edge weight 1. Choose orderings of the bricks on the top left and bottom left so the thick blue lines are sublevel sets at height $i$ and so that the bricks at height $i$ and $i+1$ are as indicated. In both the top and bottom row performing a swap (obtaining the orderings pictured on the right) decreases the weight of the sublevel set at height $i$. In both cases, the swap strictly decreases width, since for the original orderings $i$ was a local maximum and inequality (*) was strictly satisfied.}
\label{Fig: swaplemma}
\end{figure}

The proof is similar to other proofs we will encounter. The basic idea is as follows. We begin by showing that any index $i$ which is a maximum for both $\mc{O}$ and $\mc{O}'$ has $\Lambda'(i) \leq \Lambda(i)$. We then consider $i-1$ and $i+1$. Each of these may be a maximum for $\mc{O}$ but not for $\mc{O}'$; if so, it does not contribute to increased width. Each may also be a maximum for $\mc{O}'$ but not for $\mc{O}$. If such is the case, we show that $i$ is a maximum for $\mc{O}$ but not for $\mc{O}'$ and $\Lambda(i \pm 1) < \Lambda(i)$. We can then conclude that $\width(\mc{O}') \leq \width(\mc{O})$. 

\begin{proof}
Assume that (*) holds. If $(*)$ is an equality, then $\Lambda'(i) = \Lambda(i)$ and so $\Lambda'  = \Lambda$. The fact that $\width(\mc{O}') = \width(\mc{O})$ follows immediately.

Assume, therefore, that $\Lambda'(i) < \Lambda(i)$. If $i$ is a maximum for both $\mc{O}$ and $\mc{O}'$, we obtain $\width(\mc{O}')$ from $\width(\mc{O})$ by removing (one copy of) $\Lambda(i)$ and replacing it with at least one copy of $\Lambda'(i)$. It may, in fact be replaced with more than one copy, as it may now be the case that $\Lambda(i\pm1) = \Lambda'(i)$, in which case $\Lambda(i\pm1)$ are now maxima. In any case, we have replaced $\Lambda(i)$ with one or more integers strictly smaller than it and so $\width(\mc{O}') < \width(\mc{O})$. 

Suppose now that $i$ is a maximum for $\mc{O}$, but not for $\mc{O}'$. In this case we remove $\Lambda(i)$ from $\width(\mc{O})$ in the process of forming $\width(\mc{O}')$. We need to consider the possibility though that $i-1$ or $i+1$ is now a maximum for $\mc{O}'$. If $i\pm 1$ was not a maximum for $\mc{O}$, then $\Lambda(i\pm 1) < \Lambda(i)$, as $i$ is a maximum for $\mc{O}$. Consequently, even if $i\pm 1$ was not a maximum for $\mc{O}$ but is for $\mc{O}'$, we again have $\width(\mc{O}') < \width(\mc{O})$, as desired.

Finally, assume that $i$ is not a maximum for $\mc{O}$. Since $\Lambda'(i) < \Lambda(i)$, the index $i$ is not a maximum for $\mc{O}'$, either. Suppose, therefore, that $i\pm 1$ is also not a maximum for $\mc{O}$, but is for $\mc{O}'$. Since it is a maximum for $\mc{O}'$, there exists $j < i-1$ such that $\Lambda(j) < \Lambda(i-1)$ and $\Lambda$ is constant on the interval $(j,i-1]$. Since $i-1$ is not a maximum for $\mc{O}$, we must have $\Lambda(i) \geq \Lambda(i-1)$. Since $i$ is not a maximum for $\mc{O}$, we must have $\Lambda(i+1) \geq \Lambda(i)$. Thus, $\sigma(B;S_i) \geq 0$.

Since $i-1$ is a maximum for $\mc{O}'$ and $i$ is not a maximum for $\mc{O}'$, we have $\Lambda'(i) < \Lambda(i-1)$. Hence, $\sigma(B;S_{i-1}) < 0$. But then
\[
0 > \sigma(B;S_{i-1}) = \sigma(B;S_i) + 2\weight(F) \geq 0,
\]
a contradiction. Hence, if $i-1$ is a maximum for $\mc{O}'$ then it for $\mc{O}$, as well.

Now suppose that $i+1$ is a maximum for $\mc{O}'$ and not for $\mc{O}$. As in the previous argument, we can conclude that $\Lambda(i) \geq \Lambda(i+1)$ and that $\Lambda(i-1) \geq \Lambda(i)$. This implies that $\sigma(A;S_{i-1}) \leq 0$.  Also, we have that $\Lambda'(i) < \Lambda(i+1)$, and so $\sigma(A;S'_i) > 0$. Therefore, we have
\[
0 < \sigma(A;S'_i) = \sigma(A;S_{i-1}) - 2 \weight(F)  \leq 0,
\]
a contradiction. Thus, if $i+1$ is a maximum for $\mc{O}'$ then it is also a maximum for $\mc{O}$. Hence, $\width(\mc{O}') \leq \width(\mc{O})$.
\end{proof}

\begin{remark}\label{share a facet}
Observe that, as a consequence of Lemma \ref{swapping lem}, if $A$ and $B$ are shortening moves for $S_i$ and if they do not share any facets, then (*) holds and so $\width(\mc{O}') \leq \width(\mc{O})$. 
\end{remark}

\begin{definition}
If $\width(\mc{O}') \leq \width(\mc{O})$ where $\mc{O}' = \tau_i \circ \mc{O}$ for some $i \in [N-1]$, then we say that $\mc{O}$ \defn{thins} to $\mc{O}'$ and that $\tau_i$ is a \defn{legal swap}. We extend the terminology so that ``thins to'' is reflexive and transitive. Hence, if $\mc{O}$ thins to $\mc{O}'$, then there exists a sequence of legal swaps converting $\mc{O}$ to $\mc{O}'$ such that width is non-increasing under each of the transpositions.
\end{definition}

\begin{definition}\label{def: loc thin}
An ordering $\mc{O}$ is \defn{locally thin} if it does not thin to any ordering with strictly smaller width. It is \defn{thin} if there is no ordering of strictly smaller width.
\end{definition}

\subsection{Delaying and Advancing}

As Thompson observes in her context, there are times when it is advantageous to delay or advance the position of a brick in an ordering. We present two useful lemmas. They correspond to Thompson's \emph{reordering principle} \cite{Thompson}. They are similar to the principle in the thin position of knot theory that a maximum for a knot can be pushed higher and a minimum can be pushed lower. See Figure \ref{Fig: delaylemma} for an example in 2-dimensions.
\begin{figure}[ht!]
\labellist
\small\hair 2pt
\pinlabel {1} at 21 18
\pinlabel {2} at 74 31
\pinlabel 3 at 143 18
\pinlabel 4 at 200 31
\pinlabel 5 at 263 18
\pinlabel 6 at 304 31
\pinlabel 7 at 351 18
\pinlabel 8 at 418 29
\pinlabel 9 at 477 18
\pinlabel {1} at 619 18
\pinlabel {2} at 664 31
\pinlabel 3 at 737 18
\pinlabel 4 at 795 31
\pinlabel 5 at 860 18
\pinlabel 6 at 899 31
\pinlabel 7 at 952 18
\pinlabel 8 at 1017 29
\pinlabel 9 at 1073 18
\pinlabel 10 at 21 88
\pinlabel 11 at 74 68
\pinlabel 12 at 123 88
\pinlabel 13 at 200 68
\pinlabel 14 at 227 88
\pinlabel 15 at 267 86
\pinlabel 16 at 297 68
\pinlabel 17 at 335 86
\pinlabel 18 at 394 90
\pinlabel 19 at 425 68
\pinlabel 20 at 478 88

\pinlabel 10 at 617 88
\pinlabel 11 at 663 68
\pinlabel 12 at 719 88
\pinlabel 13 at 778 68
\pinlabel 14 at 824 88
\pinlabel 15 at 861 86
\pinlabel 16 at 892 68
\pinlabel 17 at 930 86
\pinlabel 18 at 985 90
\pinlabel 19 at 1024 68
\pinlabel 20 at 1075 88

\pinlabel 22 [b] at 260 103
\pinlabel 23 [b] at 335 103
\pinlabel 24 [b] at 403 103
\pinlabel 25 [b] at 478 103

\pinlabel 21 [b] at 855 103
\pinlabel 22 [b] at 928  103
\pinlabel 23 [b] at 998 103
\pinlabel 24 [b] at 1071 103

\pinlabel {21} at 123 126
\pinlabel {26} at 209 129
\pinlabel {25} at 719 129
\pinlabel {26} at 802 129

\endlabellist
\includegraphics[scale=0.43]{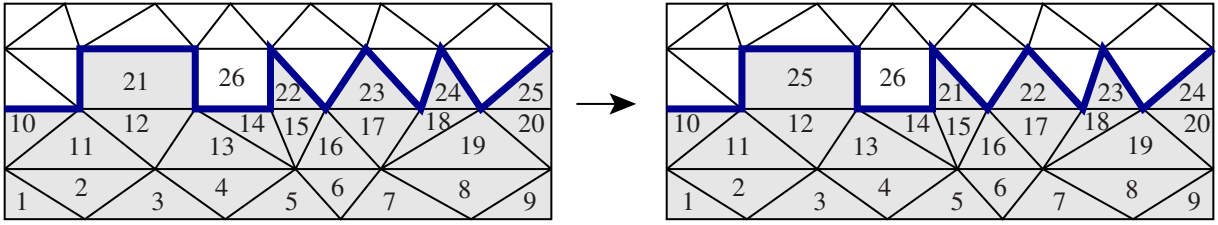}
\caption{Consider the pictured cell decomposition of a disc. Give each edge weight 1. Suppose that we have an ordering $\mc{O}$ whose first 26 triangles are as indicated. The thick blue path is the level set at height 25.  Observe that there is no local maximum between height 21 and 25. We can create a new ordering $\mc{O}'$ where we delay the brick that was at height 21 to be at height 25, shift the bricks originally lying at a height between 21 and 25, and keep all other bricks at the same height. In this example, we have not changed the heights or weights of any local maxima. We could now apply a swap the bricks at $\mc{O}'$ heights 25 and 26 to reduce the width of the ordering.}
\label{Fig: delaylemma}
\end{figure}

\begin{lemma}[The Delaying Lemma]\label{Lem:Delaying}
Let $j \in [N]$ and let  $i < j$. For the brick $A = T_i$, let $\boundary_j A = \boundary A \sqcap S_j$ and $\boundary'_j A= \boundary_I A \setminus \boundary_j A$. Suppose that
\[
\weight(\boundary_j A) > \weight(\boundary'_j A) \hspace{.3in} (\dag)
\]
Then $\mc{O}$ thins to the ordering $\mc{O}' = (i \thickspace j \thickspace j-1 \thickspace \cdots \thickspace i+1) \circ \mc{O}$.
\end{lemma}

\begin{proof}
Let $B = T_{i+1}$ and $F = A \sqcap B$. Since $M$ is a brick complex, $F$ does not share any facet with $S_j$. It also shares no facet with $S_{j-1}$. Thus, $\boundary_{S_{i-1}} A \cup F \subset \boundary'_j A$. Similarly, each facet of $\boundary_j A$ is not a facet of $S_{i-1}$. Thus, $\boundary_j A \subset \boundary'_{S_{i-1}} A \setminus F $. Consequently,
\[\begin{array}{rcl}
0 &<& \weight(\boundary_j A) -  \weight(\boundary'_j A) \\
&\leq& \weight(\boundary'_{S_{i-1}} A) - \weight(F) - \weight(\boundary_{S_{i-1}} A - \weight(F) \\
&=& \sigma(A; S_{i-1}) - 2\weight(F)
\end{array}
\]
Hence,
\[
\sigma(A;S_i) + \sigma(B;S_{i}) + 2\weight(F) = - \weight(A;S_{i-1}) + \sigma(B;S_{i}) + 2\weight(F) < \sigma(B;S_i)
\]
If $i$ is a maximum for $\mc{O}$, then $\sigma(B;S_i) \leq 0$, in which case by Lemma \ref{swapping lem}, $\width(\mc{O}') < \width(\mc{O})$, as desired. 

Suppose that $i$ is not a maximum for $\mc{O}$. Since $\sigma(A;S_{i-1}) > 2\weight(F) \geq 0$, we have $\Lambda(i) > \Lambda(i-1)$. Furthermore,
\[
\sigma(A;S'_i) = \sigma(A; S_{i-1}) - 2 \weight(F) > 0.
\]
Thus, $\Lambda(i+1) > \Lambda'(i)$. Hence, $i$ is not a maximum for $\mc{O}'$ either. 

Suppose, for the moment, that $\Lambda(i+1)$ is a maximum for $\mc{O}'$. Since $\sigma(A;S_{i-1}) > 0$, we have $\Lambda(i) > \Lambda(i-1)$. Since $i$ is not a maximum for $\mc{O}$, then $\Lambda(i+1) > \Lambda(i)$.  Hence, $i+1$ is a maximum for $\mc{O}$ as well.

Suppose that $i-1$ is a maximum for $\mc{O}'$. We will derive a contradiction. Thus, $\Lambda'(i) \leq \Lambda(i-1)$. If equality holds, then $i$ would also be a maximum for $\mc{O}'$, which it is not. Thus, $\Lambda'(i) < \Lambda(i-1)$. Thus, $\sigma(B;S_{i-1}) < 0$. Hence, $\sigma(B;S_i) < 0$. This implies $\Lambda(i+1) < \Lambda(i)$. Since $\sigma(A;S_{i-1}) > 0$, we also have $\Lambda(i) > \Lambda(i-1)$. But this implies that $i$ is a maximum for $\mc{O}$, a contradiction. Thus, $i-1$ is not a maximum for $\mc{O}'$. 

We conclude that $\width(\mc{O}') = \width(\mc{O})$ when $i$ is not a maximum for $\mc{O}$. Hence, $\mc{O}$ thins to $\mc{O}'$ as desired.

If $i < j - 1$, then performing the swap does not change the fact that ($\dag$) holds, and so we may continue performing swaps, thinning the ordering each time, until $A$ is at height $j$.
\end{proof}

We also need a version of Lemma \ref{Lem:Delaying} which allows us to advance bricks in the ordering. We could adapt the previous proof, but instead we introduce the reverse ordering. It is similar to turning a Morse function upside down.

\begin{definition}
Suppose that $\mc{O}\co \mc{F} \to [N]$ is an ordering. The \defn{reverse ordering} $r\mc{O}$ is defined by 
\[
r\mc{O}(k) = \mc{O}(N - k + 1)
\]
for each $k \in [N]^*$. 
\end{definition}

Observe that the level surface for $\mc{O}$ at height $k$ is equal to the level surface for $r\mc{O}$ at height $n-k$.  Thus, the widths of $\mc{O}$ and $r\mc{O}$ are the same. 

\begin{lemma}[The Advancing Lemma]\label{Lem:Advancing}
Let $i \in [N-2]$ and let  $j+1 < i$. For the brick $A = T_i$, let $\boundary_j A = \boundary A \sqcap S_j$ and $\boundary'_j A= \boundary_I A\setminus \boundary_j A$. Suppose that
\[
\weight(\boundary_j A) > \weight(\boundary'_j A) \hspace{.3in} (\dag)
\]
Then $\mc{O}$ thins to the ordering $\mc{O}' = ( j \thickspace j+1 \thickspace j+2 \thickspace \cdots \thickspace i) \circ \mc{O}$.
\end{lemma}
\begin{proof}
Let $B = T_{i-1}$. Consider the reverse ordering $r\mc{O}$. Let $T'_k = r\mc{O}^{-1}(k)$. Set $i' = N - i + 1$ and $j' = N - j$. Thus, $(i'+1) = N - (i-1) + 1$.  Hence, $A = T'_{i'}$ and $B = T'_{i'+1}$. Also, 
\[
r\big( \tau_{i-1} \circ \mc{O}\big) = \tau_{i'} \circ r\mc{O}.
\]

Let $S'_k = S_{N - k}$ be the level surface for $r\mc{O}$ at height $k$. Thus, $S'_{j' - 1} = S_{j}$ By hypothesis,
\[
\weight(\boundary_{S'_{j'}} A) > \weight(\boundary'_{S'_{j'}} A).
\]
Since $j + 1 \leq i-1$, we have $N - j - 1 \geq N - i + 1$. Hence, $j' = N - j \geq N-i + 2 = i' + 1$. By Lemma \ref{Lem:Delaying}, we have
\[
\width\big(\tau_{i'} \circ r \mc{O}\big) \leq \width\big(r\mc{O}\big).
\]
Hence,
\[
\width(\mc{O}') \leq \width(\mc{O}),
\]
as desired.
\end{proof}

\begin{remark}
Observe how the thinning sequences produced by the proofs of Lemmas  \ref{Lem:Delaying} and \ref{Lem:Advancing} do not move any bricks not between heights $i$ and $j$ in the ordering.
\end{remark}

\section{Local thin orderings and minimal surfaces}
 
Our main theorem states that we can use locally thin orderings to produce stable and unstable minimal surfaces. Its simplicity should be contrasted with existence results for minimal surfaces in differential geometry. It extends \cite[Theorem 6]{Thompson} to brick complexes of arbitrary dimension. 

\begin{theorem}\label{Main Theorem}
Suppose that $(M, \omega)$ is a weighted brick complex. Let $\mc{O}$ be a locally thin ordering. The following hold:
\begin{enumerate}
\item If $j$ is a maximum of $\mc{O}$, then $S_j$ is either a stable or an unstable minimal surface. If $j$ is an extremal maximum then $S_j$ is an unstable minimal surface.
\item If $j$ is a minimum of $\mc{O}$, then $ S_j$ is a stable minimal surface.
\end{enumerate}
\end{theorem}

\begin{figure}[ht!]
\labellist
\small\hair 2pt

\pinlabel 1 at 44 538
\pinlabel 2 at 94 883
\pinlabel 3 at 44 827
\pinlabel 4 at 397 883
\pinlabel 5 at 397 591
\pinlabel 6 at 337 538
\pinlabel 7 at 198 688
\pinlabel 13 at 249 733
\pinlabel 8 at 94 733
\pinlabel 9 at 249 591
\pinlabel {10} at 94 591
\pinlabel {11} at 44 688
\pinlabel {12} at 198 538
\pinlabel {14} at 686 827
\pinlabel {15} at 746 883
\pinlabel {16} at 835 827
\pinlabel {17} at 835 688
\pinlabel {18} at 898 733

\pinlabel 1 at 44 48
\pinlabel 2 at 94 403
\pinlabel 3 at 44 351
\pinlabel 4 at 397 403
\pinlabel 5 at 397 110
\pinlabel 6 at 337 48
\pinlabel 7 at 198 188
\pinlabel 13 at 249 251
\pinlabel 8 at 94 251
\pinlabel 9 at 249 110
\pinlabel {10} at 94 110
\pinlabel {11} at 44 188
\pinlabel {12} at 198 48
\pinlabel {14} at 198 827
\pinlabel {15} at 249 883
\pinlabel {16} at 337 827
\pinlabel {17} at 337 688
\pinlabel {18} at 397 733

\pinlabel 1 at 549 538
\pinlabel 2 at 603 883
\pinlabel 3 at 549 827
\pinlabel 4 at 898 883
\pinlabel 5 at 898 591
\pinlabel 6 at 835 538
\pinlabel 7 at 686 688
\pinlabel 13 at 746 733
\pinlabel 8 at 603 733
\pinlabel 9 at 746 591
\pinlabel {10} at 603 591
\pinlabel {11} at 549 688
\pinlabel {12} at 686 538
\pinlabel {14} at 198 351
\pinlabel {15} at 249 403
\pinlabel {16} at 337 351
\pinlabel {17} at 337 188
\pinlabel {18} at 397 251

\pinlabel 1 at 549 48
\pinlabel 2 at 603 403
\pinlabel 3 at 549 351
\pinlabel 4 at 898 403
\pinlabel 5 at 898 110
\pinlabel 6 at 835 48
\pinlabel 7 at 686 188
\pinlabel 13 at 746 251
\pinlabel 8 at 603 251
\pinlabel 9 at 746 110
\pinlabel {10} at 603 110
\pinlabel {11} at 549 188
\pinlabel {12} at 686 48
\pinlabel {14} at 686 351
\pinlabel {15} at 746 403
\pinlabel {16} at 835 351
\pinlabel {17} at 835 188
\pinlabel {18} at 898 251

\pinlabel {I} [b] at 5 931
\pinlabel {II} [b] at 503 931
\pinlabel {III} [t] at 6 0
\pinlabel {IV} [t] at 503 0

\endlabellist
\includegraphics[scale=0.4]{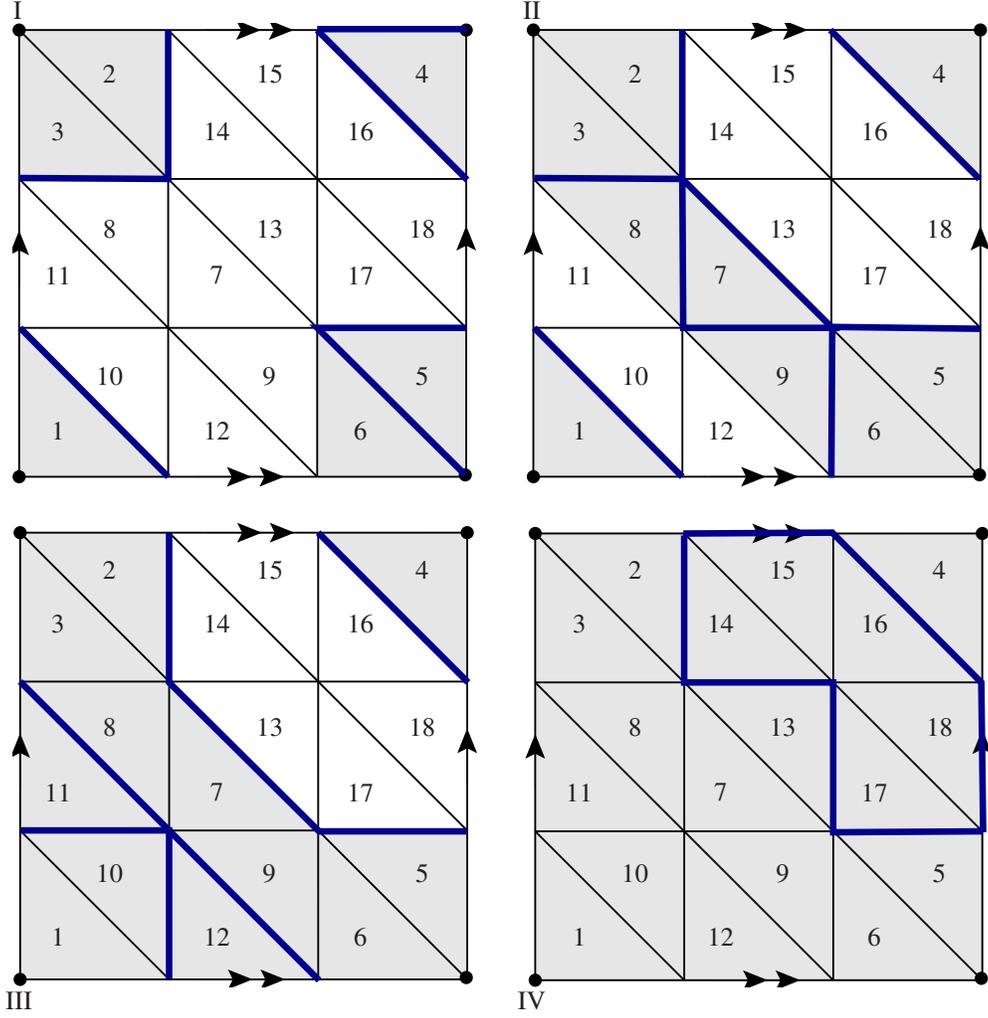}
\caption{A well-known triangulation of the torus with one vertex marked and an ordering indicated. In each frame, the thick blue curve is the level set of a local maximum of the ordering and the boundary of the shaded triangles is either empty or the level set of a local minimum of the ordering.}
\label{Fig: torustriang}
\end{figure}

Before proving Theorem \ref{Main Theorem}, we consider the example of a triangulated torus depicted in Figure \ref{Fig: torustriang}. The ordering on the triangulation is a combinatorial version of the ``standard'' Morse function on a torus with a single critical point of index 0 and a single critical point of index 2.

\begin{example}\label{torusex}
In Frame I of Figure \ref{Fig: torustriang}, $S_5$ is an unstable minimal surface. Triangles 4, 5, and 6 are all strict shortening moves. The surface $S_6$ is a stable minimal surface; it is the link of the marked vertex.  In Frame II, $S_7$ is an unstable minimal surface. Triangles 7, 8, and 9 are its strict shortening moves. The surface $S_9$ is a stable minimal surface. In Frame III, $S_{10}$ is an unstable minimal surface with triangles 10, 11, and 12 the strict shortening moves. In Frame IV, $S_{13}$ is an unstable minimal surface, with triangles 13, 14, 17 being the strict shortening moves.
\end{example}

\begin{proof}
We begin by proving the statement for maxima. Suppose that $j$ is a maximum of $\mc{O}$. If $j$ is an extremal maximum, then either $\Lambda(j-1) < \Lambda(j)$ or $\Lambda(j+1) < \Lambda(j)$. Such would imply either $\sigma(T_{j};S_j) < 0$ or $\sigma(T_{j+1};S_j) < 0$. In which case, either $T_j$ or $T_{j+1}$ is a strict shortening move. If $S_j$ has no strict shortening moves, then it is a stable minimal surface, and we are done. Suppose, therefore, that $S_j$ has a strict shortening move $T_i$.  Let $\mc{A}$ be the set of bricks of $M$ with height at most $j$ and $\mc{B}$ be the set of bricks of height at least $j+1$. By definition,
\[
\weight(\boundary_{S_j} T_i) > \weight(\boundary'_{S_j} T_i).
\]
By Lemma \ref{Lem:Delaying}, if $i < j$, then we may thin $\mc{O}$ to an ordering $\mc{O}'$  in such a way that each thinning move preserves the set $\mc{A}$ and is constant on $\mc{B}$ and so that $T_i$ is moved to be at height $j$. Since $\mc{O}$ is locally thin, width does not change. Similarly, if $i > j+1$, then by  Lemma \ref{Lem:Advancing}, we may similarly thin $\mc{O}$ to an ordering $\mc{O}'$ preserving the set $\mc{B}$, remaining constant on $\mc{A}$ and moving $T_i$ to height $j+1$. Since $\mc{O}$ is locally thin, width does not change. Indeed,  given strict shortening moves $A$ and $B$ for $S_j$ with $\mc{O}(A) \leq j$ and $\mc{O}(B) \geq j+1$, we may thin $\mc{O}$ to an ordering $\mc{O}'$ with $\mc{O}'(A) = j$ and $\mc{O}'(B) = j+1$. This may be done so that the surface $S_j$ is unchanged. Furthermore, since $A$ and $B$ are strict shortening moves, $j$ remains a maximum. 

Given a strict shortening move $A \in \mc{A}$ for $S_j$, we perform the first of our reorderings to ensure that $A$ is at height $j$. Since $j$ is a maximum, $B = T_{j+1} \in \mc{B}$ is a shortening move. If
\[
\sigma(A;S_j) + \sigma(B;S_j) + 2\weight(F) < 0,
\]
then the ordering obtained by interchanging $A$ and $B$ has strictly smaller width. The ordering $\mc{O}$ thins to this ordering, a contradiction. Thus, 
\[
\sigma(A;S_j) + \sigma(B;S_j) + 2\weight(F) \geq 0.
\]
Since both $\sigma(A;S_j)$ and $\sigma(B;S_j)$ are non-positive,
\[
 2\weight(F) \geq |\sigma(A;S_j)| + |\sigma(B;S_j)|.
\]
Similarly, if $B \in \mc{B}$ is a strict shortening move for $S_j$, then letting $A = T_j$ we have that $A\in \mc{A}$ is a shortening move for $S_j$ and
\[
 2\weight(F) \geq |\sigma(A;S_j)| + |\sigma(B;S_j)|.
\]
Finally, if $A \in \mc{A}$ and $B \in \mc{B}$ are strict shortening moves, we may thin $\mc{O}$ as above to an ordering in which $A$ is at height $j$, $B$ is at height $j+1$, and we deduce again that
\[
 2\weight(F) \geq |\sigma(A;S_j)| + |\sigma(B;S_j)|.
\]
Thus, $S_j$ is an unstable minimal surface.

Now suppose that $k$ is a minimum for $\mc{O}$ and that $A = T_i$ is a strict shortening move for $S_k$. We will derive a contradiction. Suppose, first, that $i \leq k$. We deduce that
\[
\sigma(A;S_{i-1}) > 0.
\]
Thus, $\Lambda(i-1) < \Lambda(i)$. Since $k$ is a minimum for $\Lambda$, there is a maximum $j \in [i, k)$. If $j >i$, observe that
\[
\weight(\boundary_{S_k} A) \geq \weight(\boundary_{S_{j}} A) > \weight(\boundary'_{S_{j}} A) \geq \weight(\boundary'_{S_k} A)
\]
Therefore, we may delay $A$ to height $j$, without increasing width. Furthermore, by the swap lemma (Lemma \ref{swapping lem})  we may then interchange $A = T_j$ and $B = T_{j+1}$ resulting in an ordering of strictly decreased width. This contradicts the choice of $\mc{O}$ to be locally thin. If $k < i$, we may adapt the previous argument or directly apply the previous argument to the reverse ordering to derive a contradiction.
\end{proof}

\begin{corollary}
Suppose that $(M,\weight)$ is a weighted brick complex such that there is an internal facet of $M$ on which $\weight$ is non-zero. Then $M$ contains a stable or unstable minimal surface.
\end{corollary}

\section{Invariants of pseudomanifolds and connected sums}

In this section and the next we consider $n$-dimensional pseudomanifolds (Definition \ref{def: pseudomanifold}). We assign a weight $\weight = 1$ to each $(n-1)$-face of a pseudomanifold, making it into an example of a brick complex. We continue to refer to the top-dimensional simplices as ``bricks.'' Define  $\width(M)$ to be the minimal width among all orderings of the top-dimensional simplices of $M$. An ordering is thin if its width is equal to the width of the pseudomanifold.

Given two (distinct) such pseudomanifolds  $M_1$ and $M_2$, we form a connected sum as follows. For $i = 1,2$, choose an $n$-brick $C_i \subset M_i$. Let $M_1 \# M_2$ be the result of removing $C_i$ from $M_i$ and gluing together the resulting complexes via some simplicial homeomorphism $\boundary C_1 \to \boundary C_2$. The resulting triangulated manifold is not uniquely defined (as a simplicial complex); the result depends on the choice of bricks $C_1$ and $C_2$ and gluing map. The main result of this section is:

\begin{theorem}\label{lower bound}
Consider an $n$-dimensional pseudomanifold $M = M_1 \# M_2$ with $n \geq 2$ even. Let $\weight(F) = 1$ for every facet $F$ of $M$. We have:
\[
\max (\width(M_1), \width(M_2)) \leq \width(M) 
\]
\end{theorem}

This result should be compared with \cite[Corollary 6.4]{SchSch}, where Scharlemann and Schultens prove the analogous statement for the width of the connected sum of knots in $S^3$. In \cite{BT}, Blair and Tomova show that this is the best result possible for knots. However, in \cite{TT}, Taylor and Tomova show how a modification of the definition of width for knots leads to it being additive under connected sum. In the next section, we explore possible parallels between that result and the setting of this paper. 

The reason for restricting to even dimensional manifolds is that, when $n$ is even, each $n$-simplex has an odd number of facets in its boundary. Consequently, when each facet has weight 1, then for any ordering and every index $\Lambda(i) \neq \Lambda(i \pm 1)$. This makes the study of maxima (and, therefore, width) much easier. 

\begin{proof}[proof of Theorem \ref{lower bound}]
Let $C_1$ and $C_2$ be the bricks whose interiors are removed from $M_1$ and $M_2$ to form $M = M_1 \# M_2$. Let $P$ be the homeomorphic image of their boundaries in $M$. Observe that by the definition of simplicial complex, for each facet of $P$, there is exactly one brick in $M_1$ and one brick in $M_2$ containing $P$.

 Let $X_1$ and $X_2$ be the two sides of $P$ in $M$ (so $X_i$ is obtained by removing the interior of $C_i$ from $M_i$ for $i = 1,2$). Let $\mathbb{O}$ be the set of thin orderings $\mc{O}$ on $M$. For $\mc{O}\in \mathbb{O}$, for an index $j \in [N]^*$ and $i = 1,2$, let $x_i(j)$ be the weight of the facets of $S_{j}$ which lie in $X_i$ but not in $P$. Let $T_{n_1}, T_{n_2}, \hdots, T_{n_m}$ be the subsequence of $(T_n)_{n = 1}^N$ consisting of bricks in $X_1$ and let $T_{u_1}, \hdots, T_{u_v}$ be the subsequence consisting of bricks in $X_2$. Without loss of generality, we may assume that $n_1 = 1$. (If not, exchange labels in what follows.) 
 
If it exists, let $r+1 \in [m]$ be the smallest index such that every brick of the component of $X_2$ containing $P$ is contained in $M_{n_{r+1} - 1}$. If such an $r+1$ does not exist, let $r = m$.

Consider the following orderings $\mc{O}_1$ and $\mc{O}_2$ of $M_1$ and $M_2$ respectively.
\[
\begin{array}{rcl}
\mc{O}_1&:& T_{n_1}, T_{n_2}, \hdots, T_{n_r}, C_1, T_{n_{r+1}}, \hdots, T_{n_m} \\
\mc{O}_2&:& C_2, T_{u_1}, T_{u_2}, \hdots, T_{u_v} \\
\end{array}
\]
 
We consider how the maxima of $\mc{O}$ relate to those of $\mc{O}_1$ and $\mc{O}_2$. We follow the strategy outlined before the proof of Lemma \ref{swapping lem}. Let $\Lambda_1$ and $\Lambda_2$ be the functions used in defining the widths of those orderings. Observe that for $i \leq r$:
\[
\Lambda_1(i) = \Lambda(n_i) - x_2(n_i) + \epsilon_1 - \epsilon_2
\]
where $\epsilon_1$ is the number of facets of $P$ contained in $M_{n_i}$ but not $S_{n_i}$ and $\epsilon_2$ is the number of facets of $P \cap M_{n_i}$ incident to a brick in $M_{n_i} \cap X_2$ but not to a brick of $M_{n_i} \cap X_1$. We say that $i$ and $n_i$ are \defn{corresponding} indices of $\mc{O}_1$ and $\mc{O}$ respectively. 

Also, if $i = r + j + 1$ for some $j \geq 1$, then
\[
\Lambda_1(i) = \Lambda(n_{r+j}) - x_2(n_{r+j}) - \epsilon_3 + \epsilon_4
\]
Here $\epsilon_3$ is the number of facets of $P \cap M_{n_{r+j}}$ which are either incident to a brick of $X_2 \cap M_{n_{r+j}}$ and not to a brick of $X_1 \cap M_{n_{r+j}}$ or vice versa. Also, $\epsilon_4$ is the number of facets of $P$ incident to no brick of $X_1 \cap M_{n_{r+j}}$ or $X_2 \cap M_{n_{r+j}}$. We say that $i = r+ j + 1$ and $n_{r+j}$ are \defn{corresponding} indices of $\mc{O}_1$ and $\mc{O}$ respectively. 

Finally,
\[
\Lambda_1(r+1) = \Lambda(n_r) - x_2(n_{r}) + \epsilon_5 - \epsilon_6
\]
where $\epsilon_5$ is the number of facets of $P$ which do not belong to $M_{n_r}$ and $\epsilon_6$ is the number of facets of $P$ which belong to $M_{n_r} \cap X_1$ but not $M_{n_r} \cap X_2$.

Similarly, for $i \in [v]$ we say that $i+1$ and $u_i$ are \defn{corresponding} indices of $\mc{O}_2$ and $\mc{O}$ respectively. Observe
\[
\Lambda_2(i+1) = \Lambda(u_i) - x_1(u_i) + \epsilon_7 - \epsilon_8
\]
where $\epsilon_7$ is the number of facets of $P$ which do not belong to $M_{u_i}$ and $\epsilon_8$ is the number of facets which belong to $M_{u_i} \cap X_2$ but not $M_{u_i} \cap X_1$.

\textbf{Claim 1:} Suppose that $\mu$ and $n_\ell$ are corresponding indices of $\mc{O}_1$ and $\mc{O}$ respectively. If both are maxima for their respective orderings, then $\Lambda_1(\mu) \leq \Lambda(n_\ell)$.

Suppose first, that $\ell \leq r$. Let $\epsilon_1$ be the number of facets of $P$ contained in $M_{n_\ell}$ but not $S_{n_\ell}$ and $\epsilon_2$ be the number of facets of $P$ incident to a brick in $M_{n_\ell} \cap X_2$ but not to a brick of $M_{n_\ell} \cap X_1$. For a contradiction, assume that $\Lambda_1(i) > \Lambda(n_\ell)$.

Thus, $\epsilon_1 \geq 1$. In particular, some brick of $X_2$ incident to $P$ belongs to $M_{n_\ell}$. By the definition of $r$, not every brick of the component of $X_2$  containing $P$ is contained in $M_{n_\ell}$. Hence, $x_2(n_\ell) > 0$.  Therefore, there is a facet $F \subset S_{n_\ell}$ contained in $X_2$ and not in $P$. Since $M_2$ is a simplicial complex, $F$ contains at most one $(n-2)$-dimensional face which is also contained in $P$. Each other $(n-2)$-dimensional face $\alpha$ is incident to at least one other facet $F_\alpha \neq F$ of $S_{n_\ell}$. By the definition of simplicial complex, if $\alpha \neq \beta$ then $F_\alpha \neq F_\beta$. Hence, considering these $F_\alpha$ along with $F$ we see that $S_{n_\ell}$ contains at least $n$ facets in $X_2$ which are not in $P$. See Figure \ref{Fig: Sum1} for an example when $n = 2$. Thus, $x_2(n_\ell) \geq n$. Furthermore, if $F$ does not share an $(n-2)$-dimensional face with $P$, then the inequality is strict. If the inequality is strict, $x_2(n_\ell) \geq n+1 = \weight(P) \geq \epsilon_1$. In which case, $\Lambda_1(i) \leq \Lambda(n_i)$. Thus, $x_2(n_\ell) = n$. In particular, every facet of $S_{n_\ell}$ in $X_2$ but not in $P$, has exactly one $(n-2)$-dimensional face in $P$.  Also, $\epsilon_1 = n+1 = \weight(P)$ and $\epsilon_2 = 0$, as otherwise $\Lambda_1(i) \leq \Lambda(n_\ell)$. 

\begin{figure}[ht!]
\labellist
\small\hair 2pt
\pinlabel {$X_1$} at 170 149
\pinlabel {$X_2$} at 281 266
\pinlabel {$F$} [b] at 382 67
\pinlabel {$\alpha$} [tl] at 455 64
\pinlabel {$F_\alpha$} [br] at 483 105
\pinlabel {$S_{n_\ell}$} [br] at 233 102
\endlabellist
\includegraphics[scale=0.43]{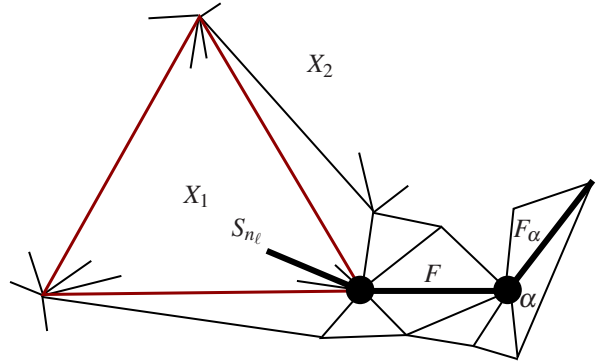}
\caption{An example of the construction of $F_\alpha$ when $n = 2$. Portions of a triangulated surface are drawn, with the surface $P$ the triangle in red. The curve $S_{n_\ell}$ is drawn in thick black lines. Two vertices of the triangulation are marked with large dots. The two marked vertices (i.e. $(n-2)$-faces) are endpoints of the facet (i.e. edge) $F \subset S_{n_\ell}$. The rightmost endpoint is $\alpha$ and  $F_\alpha$ is the edge of $S_{n_\ell}$ incident to $\alpha$.}
\label{Fig: Sum1}
\end{figure}

Let $\gamma$ be the $(n-2)$-dimensional face of $F$ contained in $P$. Since $F \subset S_{n_\ell}$, there exists a facet $\rho$ in $P$ containing $\gamma$ and incident to a brick $A_0 \subset X_2$ which does not contain $F$. Since $\epsilon_1 = \weight(P)$, the brick $A_0$ belongs to $M_{n_\ell}$.  Observe the facet $F_0$ of $\boundary A_0$ containing $\gamma$ is not equal to $F$ and cannot share an $(n-2)$-dimensional face with $F$ other than $\gamma$ (since $M$ is a simplicial complex). Since $x_2(n_\ell) = n$ and $\weight(F \cup \bigcup\limits_\alpha F_\alpha) = n$, the facet $F_0$ does not belong to $S_{n_\ell}$. Thus, there is another brick $A_1 \neq A_0$ having $\gamma$ as a face and sharing a facet with $A_0$ such that $A_1 \subset M_{n_\ell}$. There is a facet $F_1 \neq F_0$ of $A_1$ containing $\gamma$. It must also must differ from $F$ and lie in $M_{n_\ell}$. There is then another brick $A_2 \neq A_0, A_1$ in $M_{n_\ell} \cap X_2$ having $\gamma$ as a face. Continuing on in this way we define an infinite sequence of facets $F_0, F_1, \hdots$, contradicting the local finiteness of the simplicial complex $M$. See Figure \ref{Fig: Sum2} for an example when $n = 2$. 

\begin{figure}[ht!]
\labellist
\small\hair 2pt
\pinlabel {$X_1$} at 170 149
\pinlabel {$X_2$} at 281 266
\pinlabel {$F$} [t] at 436 34
\pinlabel {$\alpha$} [t] at 549 25
\pinlabel {$\gamma$} [t] at 316 26
\pinlabel {$\rho$} [r] at 251 144
\pinlabel {$A_0$} at 285 149
\pinlabel {$A_1$} at 341 94
\pinlabel {$A_2$} at 377 66
\pinlabel {$F_\alpha$} [bl] at 376 161
\pinlabel {$S_{n_\ell}$} [br] at 232 75
\endlabellist
\includegraphics[scale=0.43]{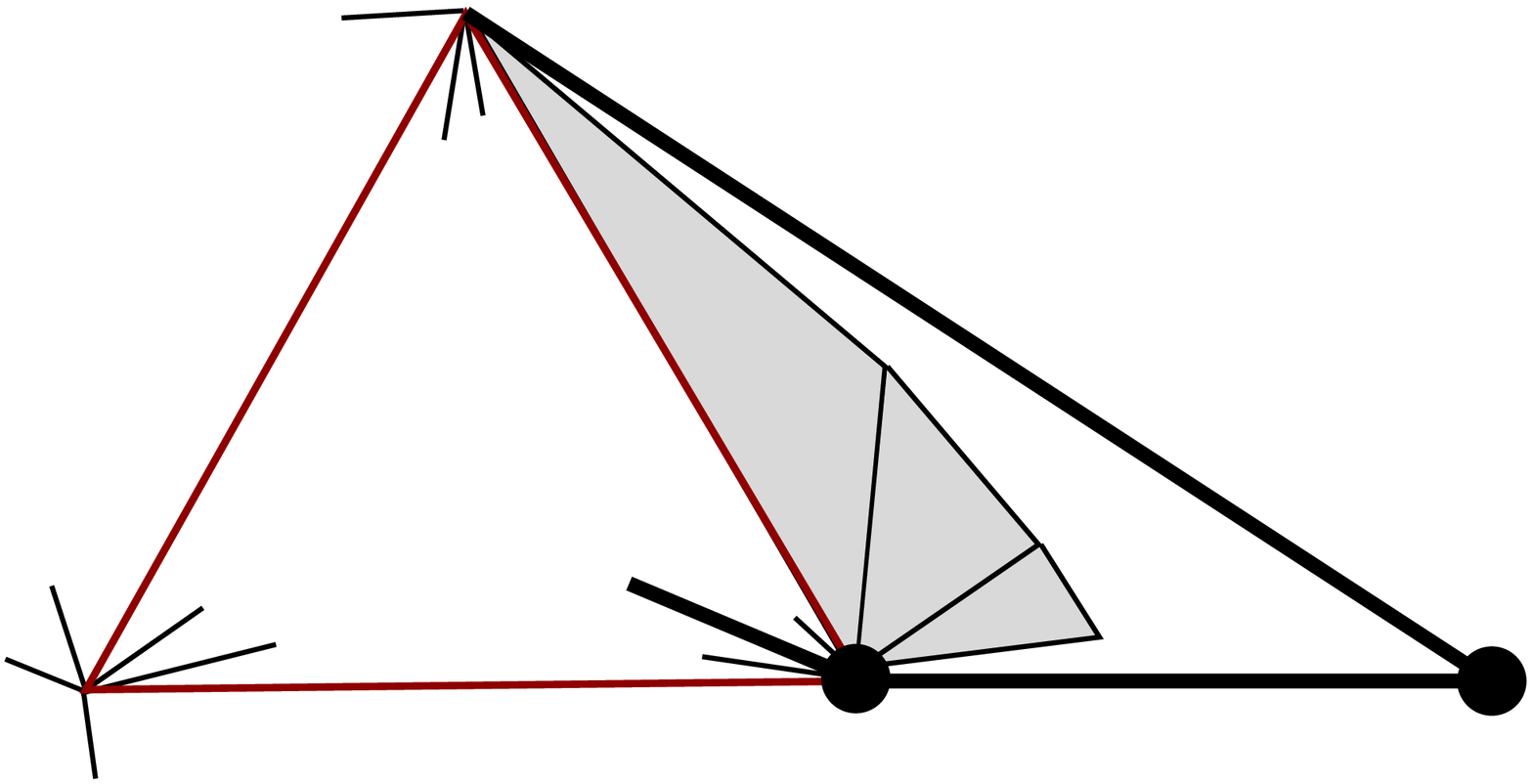}
\caption{An example of the construction of the sequence $(A_k)$ when $n = 2$. Portions of a triangulated surface are drawn, with the surface $P$ the triangle in red. The curve $S_{n_\ell}$ is drawn in thick black lines. Two vertices of the triangulation are marked with large dots. The two marked vertices (i.e. $(n-2)$-faces) are endpoints of the facet (i.e. edge) $F \subset S_{n_\ell}$. The rightmost endpoint is $\alpha$ and  $F_\alpha$ is the edge of $S_{n_\ell}$ incident to $\alpha$. The triangles $A_0, A_1, A_2, \hdots$ are shaded to indicate that they belong to $M_{n_\ell}$.}
\label{Fig: Sum2}
\end{figure}

Now suppose that $\ell \geq r+1$. By definition of $\mc{O}_1$, there is a $j \in [m]$ such that $\mu = r+j+1$ and $\ell = r+j$. Let $\epsilon_3$ is the number of facets of $P \cap M_{n_\ell}$ which are either incident to a brick of $X_2 \cap M_{n_{r+j}}$ and not to a brick of $X_1 \cap M_{n_{r+j}}$ or vice versa. Let $\epsilon_4$ is the number of facets of $P$ incident to no brick of $X_1 \cap M_{n_\ell}$ or $X_2 \cap M_{n_\ell}$. By the definition of $r$, every facet of $P$ belongs to $M_{n_r}$. Thus, $\epsilon_4 = 0$. Hence, $\Lambda_1(\mu) \leq \Lambda(n_\ell)$, as desired.

\textbf{Claim 2:} $r+1$ is not a maximum for $\mc{O}_1$.

We have $\Lambda_1(r+1) = \Lambda(r)  - x_2(n_r)+  \epsilon_5 - \epsilon_6$, where $\epsilon_5$ is the number of facets of $P$ not belonging to $M_{n_r}$ and $\epsilon_6$ is the number of facets of $P$ belonging to a brick in $M_{n_r} \cap X_1$ but not to a brick in $M_{n_r} \cap X_2$. Assume, for a contradiction, that $x_2(n_r) + \epsilon_6 \leq \epsilon_5$. If $x_2(n_r) = 0$, then, as in Claim 1, no brick of $M_{n_{r}}$ lies in $X_2$. This contradicts, the definition of $r+1$. Thus, $x_2(n_r) \geq 1$. As in Claim 1, if there is a facet $F$ of $S_{n_r} \cap X_2$ not sharing an $(n-2)$-face with $P$, then 
\[
x_2(n_r) \geq n+1 \geq \weight(P) \geq \epsilon_5.
\]
Thus, every facet $F$ of $S_{n_r} \cap X_2$ shares exactly one $(n-2)$-face with $P$. (It cannot share more than one since $M_2$ is a simplicial complex.) As before, we may also conclude that $x_2(n_r) = n$ that $\epsilon_6 = 0$, and that $\epsilon_5 = n+1$. Thus, no facet of $P$ belongs to $M_{n_r}$. Let $\gamma$ be an $(n-2)$-dimensional face of $P$ which is also a face of some brick $A \subset M_{n_r} \cap X_2$. The face $\gamma$ is a face of two distinct facets $F$ and $F'$ of $\boundary A$. Neither lies in $P$. Without loss of generality, we may assume that $F$ lies in $S_{n_r}$. The other $n-1$ faces of dimension $(n-2)$ of $F$  (besides $\gamma$) are incident to distinct facets of $S_{n_r}$, none of which is $F'$.  If $F'$ was in $S_{n_r}$, then we would have $x_2(n_r) \geq n+1$, a contradiction. Thus, $F'$ does not lie in $S_{n_r}$. The brick $A_1$ adjacent to $A$ across $F'$ must, therefore, lie in $M_{n_r}$. Let $F_1 \neq F'$ be the other face of $A_1$, containing $\gamma$. As before, $F_1$ cannot lie in $S_{n_r}$. Continuing on in this way, we construct an infinite sequence of facets having $\gamma$ as a face. This contradicts the fact that $M$ is a simplicial complex.

\textbf{Claim 3:} Suppose $\mu$ and $n_i$ are corresponding indices of $\mc{O}_1$ and $\mc{O}$ such that $\mu$ is a maximum for $\mc{O}_1$ but $n_i$ is not a maximum for $\mc{O}_2$, then there exists a maximum $s$ for $\mc{O}$ in the interval $[n_i + 1, n_{i+1} - 1]$ such that $\Lambda(s) > \Lambda_1(\mu)$.

Suppose, first, that $i \leq r$, so that $\mu = i$. Since $\mu$ is a maximum for $\mc{O}_1$, $\Lambda_1(i) > \Lambda_1(i-1)$. This implies that $\Lambda(n_i) > \Lambda(n_{i-1})$. If $n_{i-1} = n_i - 1$, then we also have $\Lambda(n_i) > \Lambda(n_i - 1)$. If $n_{i-1} < n_i - 1$, then $T_{n_i}$ is a brick in $X_1$ and $T_{n_i - 1}$ is a brick in $X_2$. They share at most one facet (which, if it exists, is in $P$). Let $\zeta$ be the number of facets shared between $T_{n_i}$ and $M_{n_{i-1}} \cap X_1$. Notice that $\zeta \leq n/2$. Since $T_{n_i}$ has $n+1$ facets, we have
\[
\Lambda(n_i) - \Lambda(n_i - 1) = (n+1) - \zeta - 1 \geq n/2 \geq 1.
\]
Thus, $\Lambda(n_i) > \Lambda(n_i - 1)$, in this case also.

Since $n_i$ is not a maximum for $\Lambda$, we must have $\Lambda(n_i + 1) > \Lambda(n_i)$. Suppose that $i \leq r-1$.  Then this implies $n_{i+1} > n_i + 1$. Since $\Lambda(n_{i+1}) < \Lambda(n_i)$, there must be a maximum $s \in [n_i + 1, n_{i+1} - 1]$ for $\Lambda$ with $\Lambda(s) > \Lambda(n_i)$, as desired. If $i = r$, then by the definition of $r$, we again have $n_{r+1} > n_r + 1$. The same argument then produces our desired maximum. 

Suppose now that $i \geq r+1$. The argument is similar to what we had before. We have $\mu = i+1$. Since $i$ is a maximum for $\Lambda_1$, $\Lambda(n_i) > \Lambda(n_{i+1})$. If $i \geq r+2$, then for the same reason we also have $\Lambda(n_i) > \Lambda(n_{i-1})$. If $n_{i-1} = n_i -1$, then tautologically $\Lambda(n_i) > \Lambda(n_i - 1)$. If $n_{i-1} < n_i - 1$, then $T_{n_i}$ is a brick in $X_1$, but $T_{n_i - 1}$ is a brick in $X_2$. These bricks share at most one facet in common, so again $\Lambda(n_i) > \Lambda(n_i - 1)$. Since $i$ is not a maximum for $\Lambda$, $\Lambda(n_i + 1) > \Lambda(n_i)$. Since $\Lambda(n_{i+1}) < \Lambda(n_i)$, there is a maximum $s \in [n_i + 1, n_{i+1}  - 1]$ for $\Lambda$ with $\Lambda(s) > \Lambda(n_i)$. 

\textbf{Claim 4:} $\width(\mc{O}_1) \leq \width(\mc{O})$

This follows immediately from Claims 1 - 3.

\textbf{Claim 5:} Suppose that $i+1$ and $u_i$ are corresponding indices of $\mc{O}_2$ and $\mc{O}$ respectively. If both are maxima for their respective orderings, then $\Lambda_2(i+1) \leq \Lambda(u_i)$.

This is nearly identical to Claim 1. Let $\epsilon_7$ be the number of facets of $P$ which do not belong to $M_{u_i}$ and $\epsilon_8$ be the number of facets which belong to $M_{u_i} \cap X_2$ but not $M_{u_i} \cap X_1$. Let $W = M_{u_i} \cap X_1$. An argument similar to the previous ones shows that  $x_1(u_i) \geq \epsilon_7$. Thus, $\Lambda_2(i+1) \leq \Lambda(u_i)$.

\textbf{Claim 6:} $1$ is not a maximum for $\mc{O}_2$.

We have $\Lambda_2(1) = \weight(P) = n+1$. Each brick of $X_2$ is incident to at most one facet of $C_2$. Thus, either $\Lambda_2(2) = \Lambda_2(1) + (n+1)$ or $\Lambda_2(2) = \Lambda(1) + n- 1$. Thus, in either case $\Lambda_2(2) > \Lambda_2(1)$.

\textbf{Claim 7:}  Suppose $\mu = i+1$ and $u_i$ are corresponding indices of $\mc{O}_2$ and $\mc{O}$ such that $\mu$ is a maximum for $\mc{O}_2$ but $u_i$ is not a maximum for $\mc{O}_2$, then there exists a maximum $s$ for $\mc{O}$ in the interval $[u_i + 1, u_{i+1} - 1]$ such that $\Lambda(s) > \Lambda_2(\mu)$.

Assume that $\Lambda_2(i+1)\geq \Lambda_2((i+1) \pm 1)$. Consequently, for $i \geq 1$, we have $\Lambda(u_i) > \Lambda(u_{i+1})$ and for all $i \geq 2$, we have $\Lambda(u_i) \geq \Lambda(u_{i-1})$. If, for $i \geq 2$, we have  $u_{i-1} = u_i - 1$, then $\Lambda(u_i - 1) < \Lambda(u_i)$, tautologically. If, for $i \geq 2$, we have $u_{i-1} < u_i - 1$, then $T_{u_i - 1}$ lies in $X_1$ and $T_{u_i}$ lies in $X_2$. Since every facet that $T_{u_i}$ shares with $X_1$ is also shared with $B$, we have $\Lambda(u_i) > \Lambda(u_i - 1)$. Since $u_i$ is not a maximum for $\mc{O}$, we must have $\Lambda(u_i + 1) > \Lambda(u_i)$. Since $\Lambda(u_{i+1}) < \Lambda(u_i)$, there exists $s \in [u_i + 1, u_{i+1} - 1]$ such that $s$ is a maximum for $\Lambda$ and $\Lambda(s) > \Lambda(u_i)$. 

\textbf{Claim 8:} $\width(\mc{O}_2) \leq \width(\mc{O})$
This follows immediately from Claims 5 - 7.
\end{proof}

For widths $\width(V_1)$ and $\width(V_2)$, let $\width(V_1) \cup \width(V_2)$ be the sequence with the same terms as the sequences $\width(V_1)$ and $\width(V_2)$, but arranged in non-increasing order. As a consequence, we have:

\begin{corollary}\label{width bounds}
Suppose that $V_1$ and $V_2$ are two closed $n$-dimensional simplicial pseudomanifolds with $n\geq 2$ even. Then there is a connected sum $V_1 \# V_2$ such that
\[
\max(\width(V_1), \width(V_2)) \leq \width(V_1 \# V_2) \leq \width(V_1) \cup \width(V_2)
\]
\end{corollary}
\begin{proof}
Suppose that 
\[
\mc{O}_1: T_1, T_2, \hdots, T_a
\]
and 
\[
\mc{O}_2: U_1, U_2, \hdots, U_b
\] 
 are thin orderings of $V_1$ and $V_2$ respectively. Let $M = V_1 \# V_2$ be the $n$-manifold obtained by removing the interiors of $T_a$ and $U_1$ and choosing a simplicial homeomorphism $\boundary T_a \to \boundary U_1$ to glue the resulting spaces together. Consider the ordering
 \[
 \mc{O}: T_1, T_2, \hdots , T_{a-1}, U_2, \hdots, U_{\ell -1}, U_{\ell +1}, \hdots, U_b.
 \]
of $M$. Observe that for $i \leq a-2$, $i$ is a maximum for $\mc{O}$ if and only if it is a maximum for $\mc{O}_1$. Furthermore, for such an $i$, $\Lambda(i) = \Lambda_1(i)$. For $i \geq a+1$,  since $\boundary U_1$ and $\boundary M_{a-1}$ contain exactly the same facets, $i$ is a maximum for $\mc{O}$ if and only if $i-a+2$ is a maximum for $\mc{O}_2$. Furthermore, the maxima take exactly the same values.

Consider $i= a-1$. We have $\Lambda_1(i) > \Lambda_1(i+1)$, so if $i$ is a maximum for $\mc{O}$, it is also a maximum for $\mc{O}_1$ of exactly the same value.  Now consider $i = a$. Since $U_1$ and $U_2$ share at most one facet by the definition of simplicial complex, $\Lambda_2(2) > \Lambda_2(1)$. Consequently, if $i = a$ is a maximum for $\mc{O}$, then $i - a + 2$ is also a maximum for $\mc{O}_2$ of exactly the same value. 

Hence, \[\width(\mc{O}) \leq \width(\mc{O}_1) \cup \width(\mc{O}_2) = \width(M_1) \cup \width(M_2).\]

From Theorem \ref{lower bound}, we see that
\[
\max(\width(M_1), \width(M_2)) \leq \width(M). 
\]
\end{proof}

In \cite{Ozawa}, Ozawa defines an invariant of knots in $S^3$, called \emph{trunk}. He shows that the trunk of the connected sum of two knots is at most the maximum of the trunk of the factors. Problem 1.8 of \cite{Ozawa} asks if equality necessarily holds. In \cite{DZ}, Davies and Zupan show that it does. We adapt Ozawa's definition to our setting and prove the corresponding theorem.

\begin{definition}[cf. \cite{Ozawa}]
Suppose that $M$ is an $n$-dimensional pseudomanifold and let $\mc{O}$ be an ordering of $M$. The \defn{trunk} of $\mc{O}$ is equal to the maximum value  $\tr(\mc{O})$ achieved by $\Lambda$. The \defn{trunk}  $\tr(M)$ of $M$ is the minimum of $\tr(\mc{O})$ over all orderings $\mc{O}$ of $M$.
\end{definition}

Recall that for an ordering $\mc{O}$,  $\tr(\mc{O})$ is the initial entry in the sequence $\width(\mc{O})$. Since we compare width lexicographically and since thin orderings minimize width, $\tr(M) = \tr(\mc{O})$ for every thin ordering $\mc{O}$. An easy adaptation of the proof of Theorem \ref{width bounds} yields the following version of the Davies-Zupan theorem.

\begin{corollary}
Let $M_1$ and $M_2$ be an $n$-dimensional pseudomanifolds with $n \geq 2$ even. Then there exists a connected sum $M_1 \# M_2$ such that
\[
\tr(M_1 \# M_2) = \max(\tr(M_1), \tr(M_2))
\]
\end{corollary}

\section{Questions and Conjectures}

\subsection{Comparisons with Thompson's work}
Theorem 13 of \cite{Thompson} shows that (using Thompson's definitions) if a 2-dimensional sphere has a triangulation with a thin ordering having a single maximum, then the triangulation is the tetrahedral triangulation. Recall that in our context, each brick in an $n$-dimensional pseudomanifold is an $n$-simplex and that we traditionally assign a weight of 1 to each $(n-1)$-face (i.e. facet).

\begin{question}
Suppose that $M$ is an $n$-dimensional pseudomanifold such that there is a thin ordering having a single maximum. What special topological properties does $M$ have? Must it be an $n$-ball?
\end{question}

Thompson \cite{Thompson} uses the fact that all triangulations of a disc are shellable to construct \emph{embedded} closed geodesics in triangulated 2-spheres. In our context, this raises the question:

\begin{question}
Do the stable and unstable combinatorial minimal surfaces constructed in this paper correspond in some natural way to \emph{embedded} stable and unstable minimal surfaces?
\end{question}

Section 6 of \cite{Thompson} proves a combinatorial analogue of the ``3 geodesics theorem'' from differential geometry. For example, considering the tetrahedron as a 2-dimensional sphere, if we choose a pair of disjoint edges, the remaining edges are an unstable geodesic (i.e. minimal surface). As there are three such pairs of disjoint edges, the tetrahedron has three unstable geodesics. If a sphere is triangulated with more vertices, Thompson's construction produces an ordering with at least two maxima and one minima. Consequently, there are at least two unstable geodesics and one stable geodesic. As the author pointed out to us \cite{Thompson-pc}, the result is somewhat weak since generally the link of a vertex in a triangulated 2-sphere will be a stable geodesic. This means that it is usually easy to find at least 3 closed stable geodesics without appealing to her techniques. In our view, this should not cause too much distress, as a triangulated sphere with many vertices should correspond to a very bumpy sphere, where we might expect many simple closed geodesics.  Nevertheless, we can ask:

\begin{question}
If $P$ is a triangulated closed 2-dimensional manifold (every edge having weight 1), what is the relationship between the genus of $P$, the number of vertices of the triangulation, and the number of closed, embedded combinatorial geodesics?
\end{question}

\begin{example}
Consider the triangulation of the torus given in Figure \ref{Fig: torustriang}. It is given by a decomposition of a square into 9 smaller squares and then decomposing each of those into two triangles. The three vertical lines and the three horizontal lines and the three diagonals are all stable geodesics. However, there are other stable geodesics, corresponding to some of the other slopes. Each square is also a stable geodesic. 
\end{example}

\subsection{Comparison with Bachman's work}

We can try to interpret stable and unstable minimal surfaces using Bachman's theory of topological index. For simplicity, we phrase this only for embedded surfaces (i.e. curves) in a compact 2-dimensional manifold with a triangulation where every edge has weight 1. Notice that in such a setting, every shortening move is a strict shortening move. Given a surface  (i.e. curve) $S \subset M$, let $\mc{D}$ denote the simplicial complex defined as follows. Each shortening move for $S$ is a vertex of $\mc{D}$.  If we have $m\geq 2$ distinct shortening moves for $S$, they span an $m-1$ simplex if they pairwise do not share any facets of $S$. The \defn{topological index} of $S$ is $0$ if $\mc{D}$ is empty and is  $k \geq 1$, if  $k$ is the least number such that the $k$th homotopy group $\pi_{k-1}(\mc{D})$ is nontrivial. 

\begin{example}
Consider the triangulation and ordering in  Figure \ref{Fig: torustriang}.  The curve $S_5$ is an embedded, separating surface (i.e. curve) in the torus. In Example \ref{torusex} we examined the shortening moves and saw that it is an unstable minimal surface. Triangles 4, 5, and 6 are shortening moves. Since triangles 4 and 6 do not share any facets (i.e. edges) in $S_5$, they are the endpoints of an edge in $\mc{D}$. However, both of them do share a facet with triangle 5 that lies in $S_5$, and so neither spans an edge in $\mc{D}$ with triangle 5. Consequently, $\mc{D}$ consists of a single edge and an isolated vertex. In particular $\pi_0(\mc{D})$ is nontrivial. Since  $\mc{D}$ is nonempty, the topological index of $S_5$ is equal to 1.
\end{example}

\begin{lemma}
Assume that $M$ is a triangulated (as a simplicial complex) compact, connected 2-dimensional manifold, where every edge of the triangulation has weight 1. Let $S$ be a proper embedded separating surface (i.e. curve). Then $S$ is a stable minimal surface if and only if it is index 0 and $S$  is an unstable minimal surface if and only if it has topological index 1. 
\end{lemma}
\begin{proof}
The condition ``$S$ has index 0'' is equivalent to $S$ having no shortening moves. Since in our context, all shortening moves are strict shortening moves, this is equivalent to $S$ being a stable minimal surface.

Suppose, therefore, that $S$ is an unstable minimal surface. Let $\mc{A}$ and $\mc{B}$ be the partition of the bricks of $M$ from the definition. Since all shortening moves are strict shortening moves, by part (2) of Definition \ref{def:min surf}, $\mc{A}$ and $\mc{B}$ each contain a strict shortening move for $S$, so $\pi_0(\mc{D})$ is nonempty. By part (4), every strict shortening move in $\mc{A}$ shares a facet with every strict shortening move in $\mc{B}$. If $\pi_0(\mc{D})$ were trivial, then $\mc{D}$ would be connected. In particular, given $ A \in \mc{A}$ and $B \in \mc{B}$, there would be a sequence
\[
A = T_0, T_1, \hdots, T_n = B
\]
such that $T_i$ and $T_{i+1}$ are the endpoints of an edge in $\mc{D}$. Since $\mc{A}$ and $\mc{B}$ partition the triangles, we can choose $A$ and $B$ to be the endpoints of an edge in $\mc{D}$. By the definition of $\mc{D}$, this means that $A$ and $B$ do not share any facets of $S$. Since they are each shortening moves for $S$ and are triangles,
\[
|\sigma(A;S)| + |\sigma(B;S)| \in \{2, 4, 6\},
\]
with the sum being equal to 2 if and only if both $A$ and $B$ have exactly two of their edges lying in $S$. Notice also, that $|\boundary A \sqcap \boundary B| \in \{0,1\}$. By the definition of simplical complex, $A$ and $B$ can share at most one edge. By part (4) of Definition \ref{def:min surf},
\[
2 \geq 2\weight(\boundary A \sqcap \boundary B) \geq |\sigma(A;S)| + |\sigma(B;S)| \geq 2.
\]
Thus, $A$ and $B$ share exactly one edge and that edge is not in $S$. But this contradicts part (1) of Definition \ref{def:min surf}. Thus, $\pi_0(\mc{D})$ must be nontrivial, so $S$ has topological index equal to 1.

Suppose, on the other hand, that $\pi_0(S)$ is non-trivial. Let $A$ and $B$ be bricks of $M$ lying in distinct components of $\mc{D}$. Since both are shortening moves, they each have two or three of their edges lying in $S$. Since they lie in distinct components of $\mc{D}$, they are not the endpoints of an edge in $\mc{D}$. Thus, they must share an edge of the triangulation and that edge must be a facet in $S$. By the definition of simplicial complex, that is their unique shared edge. Thus,
\[
2\weight(\boundary A \sqcap \boundary B) = 2|\boundary A \sqcap \boundary B| = 2 = |\sigma(A; S)| + |\sigma(B; S)|.
\]
Since $S$ is separating, $A$ and $B$ lie on opposite sides of $S$.  We conclude that whenever $A', B' \in \mc{D}$ lie in different components of $\mc{D}$, they lie on distinct sides of $S$ and the inequality from Conditions (3) and (4) of Definition \ref{def:min surf} holds for $A'$ and $B'$. 

Let $\mc{A}$ be all the triangles of $M$ on the same side of $S$ as $A$ and $\mc{B}$ be all the triangles of $M$ on the same side of $S$ as $B$. Since $S$ is separating, Condition (1) of Definition \ref{def:min surf} is satisfied. Since $A \in \mc{A}$ and $B \in \mc{B}$, Condition (2) is satisfied. Suppose that $A'$ is a shortening move for $S$ on the same side of $S$ as $A$. If $A$ and $A'$ share no facet that is also a facet in $S$, then $A$ and $A'$ are the endpoints of an edge in $\mc{D}$, and thus lie in the same component of $\mc{D}$. If $A$ and $A'$ do share a facet, by the definition of simplicial complex, that is their unique shared facet. If that facet lies in $S$, then $A$ and $A'$ do not lie on the same side of $S$ as each other, as $S$ is an embedded curve in a 2-dimensional manifold. Thus, if $A' \in \mc{A}$ is a shortening move for $S$, then $A'$ is in the same component of $\mc{D}$ as $A$. Similarly, if $B' \in \mc{B}$ is a shortening move for $S$, then $B'$ lies in the same component of $\mc{D}$ as $B$. We conclude that Definition \ref{def:min surf} is satisfied and $S$ is an unstable minimal surface.
\end{proof}

\begin{question}
What is the ``right'' definition of index when $S$ is non-separating or non-embedded? Are there minimal surfaces of arbitrarily high index?
\end{question}

\subsection{Connected Sums}

In the previous section, we gave bounds on the width of a connect sum in terms of the widths of the factors. The upper bound relied on constructing the sum using particular choices of bricks. Perhaps the choice can be made in such a way that the upper bound no longer holds?

\begin{conjecture}
There exists $n$-dimensional simplicial manifolds $M_1$ and $M_2$ and a choice of connected sum $M_1 \# M_2$ such that
\[
\width(M_1 \# M_2) > \width(M_1) \cup \width(M_2).
\]
\end{conjecture}

For many years, it was an open question as to whether Gabai's width for knots in $S^3$ satisfied an additivity property under connected sum. Blair and Tomova \cite{BT} answered the question negatively. But in \cite{TT}, Taylor and Tomova slightly modified the definition of width so as to make it additive under connected sum of (3-manifold, graph) pairs. 

\begin{conjecture}
There is a modification of Thompson's width, so that for any connected sum of simplicial $n$-manifolds $M_1$ and $M_2$, we have
\[
\width(M_1 \# M_2) = \width(M_1) \cup \width(M_2)
\]
\end{conjecture}

As further evidence for this (admittedly vague) conjecture, we remark that in Taylor and Tomova's construction, height functions on 3-manifolds (in particular, $S^3$) are allowed to have multiple minima and that this is crucial (see \cite[Section 6]{TT}) to making width additive. In a different direction, Bowman, Heistercamp, and Johnson \cite{BHJ} were able to improve Johnson's clustering algorithm by using a version of width based on partial orders, rather than linear orders. Perhaps there is a version of the complexity in that paper which would be useful?

Finally, even in our setting, we have not addressed the case of odd dimensional manifolds.

\begin{conjecture}
If $M_1$ and $M_2$ are closed odd-dimensional triangulated manifolds, then
\[
\max(\width(M_1), \width(M_2)) \leq \width(M_1 \# M_2)
\]
for every connected sum.
\end{conjecture}

\subsection{Extensions}

\begin{question}
Is there a version of Theorem \ref{Main Theorem} for weighted simplicial complexes? or polytopal complexes?
\end{question}

It seems likely that to adapt this paper from pseudomanifolds to the setting of (not necessarily pure) simplicial complexes, we would need to take into account the number of $n$-simplices incident to each $(n-1)$-face. One natural choice of complexity is as follows. Let $\mc{O}$ be an ordering of the $n$-simplices (bricks) of $M$. Let $M_i$ be the union of all bricks $T$ such that $\mc{O}(T) \leq i$. For a facet $F$, let $\deg(F;M_i)$ be the number of bricks in $M_i$ having $F$ as a facet and let $\deg(F;M^C_i)$ denote the number of bricks not in $M_i$ having $F$ as a facet. Let $\Lambda(i) = \sum\limits_F \deg(F;M_i)\deg(F;M^C_i)\weight(F)$ where the sum is over all facets in $M$. Let $S_i$ be the union of all facets $F$ with $\deg(F;M_i)\deg(F;M^C_i) \neq 0$. Define $\width(\mc{O})$ as before, but using this $\Lambda$. Observe that when $\deg(F;M) = 2$ for every facet $F$, then $M$ is a brick complex and the new definitions for $\Lambda$ and $S_i$ coincide with the previous definitions. Some of the results of this paper can be extended to this new setting, but the fundamental difficulty in obtaining completely satisfactory answers is finding the ``right'' notion of what it means to vary a surface across a brick. 

A Pachner move on a triangulation is one of a set of local moves converting the triangulation into another triangulation.  The exact definition of the set of moves depends on the dimension of the simplicial complex. It is known that any two PL-homeomorphic triangulated manifolds are related by a sequence of Pachner moves. See \cite{Lickorish}. One of the Pachner moves is often called a \defn{stabilization}. Given an $n$-simplex $A$ in an $n$-dimensional triangulated manifold $M$, insert a new vertex $x$ in the barycenter of $A$, remove $A$, and insert $(n+1)$-simplices. Each new simplex has $n$ vertices lying in a facet of $\boundary A$ as well as having $x$ as a vertex. Let $M'$ be the new triangulated manifold. (The manifolds $M$ and $M'$ are homeomorphic, though the triangulations differ.) If $\mc{O}$ is an ordering on $M$, there is a natural ordering $\mc{O}'$ given by replacing $A$ with the new bricks $A_1, \hdots, A_{n+1}$ and then shifting the subsequent bricks of $M$ by $n$. Different choices of orderings on the new bricks of course give different orderings $\mc{O}'$. We call any of these new orderings a \defn{stabilization} of $\mc{O}$.

The following suggests an analogue of the Reidemeister-Singer theorem for Heegaard splittings of 3-manifolds.

\begin{question}
Given orderings $\mc{O}_1$ and $\mc{O}_2$ of a triangulated manifold $M$, does there exist a sequence of stabilizations of $\mc{O}_1$ and $\mc{O}_2$ to arrive at triangulated manifold $M'$ with orderings $\mc{O}'_1$ and $\mc{O}'_2$ respectively, such that there is an ordering $\mc{O}$ of $M'$ which thins to both $\mc{O}'_1$ and $\mc{O}'_2$?
\end{question}

\end{document}